\documentclass[11pt,reqno]{amsart}

\usepackage{amsmath,amssymb,amsthm,mathtools,mathrsfs}
\usepackage{enumitem,array,booktabs,tabularx,microtype,aliascnt}
\usepackage{placeins,needspace}
\usepackage[hidelinks]{hyperref}
\usepackage[capitalise,nameinlink]{cleveref}
\allowdisplaybreaks

\newtheorem{theorem}{Theorem}[section]
\newaliascnt{proposition}{theorem}
\newtheorem{proposition}[proposition]{Proposition}
\aliascntresetthe{proposition}

\newaliascnt{lemma}{theorem}
\newtheorem{lemma}[lemma]{Lemma}
\aliascntresetthe{lemma}

\newaliascnt{corollary}{theorem}
\newtheorem{corollary}[corollary]{Corollary}
\aliascntresetthe{corollary}

\theoremstyle{definition}
\newaliascnt{definition}{theorem}
\newtheorem{definition}[definition]{Definition}
\aliascntresetthe{definition}

\newaliascnt{example}{theorem}
\newtheorem{example}[example]{Example}
\aliascntresetthe{example}

\newaliascnt{question}{theorem}
\newtheorem{question}[question]{Question}
\aliascntresetthe{question}

\theoremstyle{remark}
\newaliascnt{remark}{theorem}
\newtheorem{remark}[remark]{Remark}
\aliascntresetthe{remark}
\theoremstyle{plain}

\crefname{theorem}{Theorem}{Theorems}
\crefname{proposition}{Proposition}{Propositions}
\crefname{lemma}{Lemma}{Lemmas}
\crefname{corollary}{Corollary}{Corollaries}
\crefname{definition}{Definition}{Definitions}
\crefname{example}{Example}{Examples}
\crefname{question}{Question}{Questions}
\crefname{remark}{Remark}{Remarks}
\crefname{section}{Section}{Sections}
\crefname{subsection}{Section}{Sections}

\numberwithin{equation}{section}

\newcommand{\ch}{\operatorname{ch}}
\newcommand{\CF}{\operatorname{CF}}
\newcommand{\rad}{\operatorname{rad}}
\newcommand{\maj}{\operatorname{maj}}
\newcommand{\conv}{\operatorname{conv}}
\newcommand{\vol}{\operatorname{vol}}
\newcommand{\reg}{\operatorname{reg}}
\newcommand{\Lie}{\operatorname{Lie}}
\newcommand{\D}{\mathcal D}
\newcommand{\Rspace}{\mathscr R}
\newcommand{\Ccone}{\mathscr C}
\newcommand{\Kcone}{\mathscr K}
\newcommand{\Vlat}{\mathscr V}
\newcommand{\Flattice}{\mathscr L^{\mathrm F}}
\newcommand{\Bpart}{\mathscr B}
\newcommand{\Ldim}{L^{\mathrm{dim}}}
\newcommand{\one}{\mathbf 1}
\newcommand{\eps}{\varepsilon}
\newcommand{\sq}{\mathrm{sq}}

\newcommand{\set}[1]{\left\{#1\right\}}

\newcommand{\inner}[2]{\left\langle #1,#2\right\rangle}
\title[The Cyclic-Induction Schur Cone]{The Cyclic-Induction Schur Cone:
Boolean Sums, Ramanujan-Square Positivity, and Integral Structure}

\author[Y.-T. Oh]{Young-Tak Oh}
\address{Department of Mathematics and Institute for Mathematical and Data Sciences\\
Sogang University, Seoul 04107, Republic of Korea}
\email{ytoh@sogang.ac.kr}

\subjclass[2020]{Primary 05E10; Secondary 20C30, 52B20, 05E05, 11A25}
\keywords{Schur positivity, cyclic induction, Foulkes characters,
Ramanujan sums, divisor lattices, lattice polytopes}
\date{}

\hypersetup{
  pdftitle={The Cyclic-Induction Schur Cone: Boolean Sums, Ramanujan-Square Positivity, and Integral Structure},
  pdfauthor={Young-Tak Oh},
  pdfsubject={Boolean Q-sums, Ramanujan-square functions, normalized Schur-positive regions, and integral structure},
  pdfkeywords={Schur positivity, cyclic induction, Foulkes characters,
  Ramanujan sums, divisor lattices, lattice polytopes}
}

\begin{document}
\begin{abstract}
We study the Schur-positive cone in
$\Rspace_{n,\mathbb R}\coloneqq
\operatorname{span}_{\mathbb R}\{p_d^{n/d}:d\mid n\}$ through its basis
$Q_{n,d}\coloneqq\ell_{n/d}^{(1)}[p_d]$, where $\ell_m^{(1)}$ is the
Frobenius characteristic of the representation induced to $S_m$ from a
faithful linear character of the subgroup generated by an $m$-cycle;
brackets denote plethysm.
A Boolean $Q$-sum is a sum of distinct elements of this basis.

We give a unified proof of four conjectures of Sundaram on Schur positivity
by classifying all Schur-positive Boolean $Q$-sums; the case of sums over
divisors up to a prescribed bound recovers Hou's theorem.
Specifically, for a nonempty set
$J$ of divisors of $n$, the sum $\sum_{d\in J}Q_{n,d}$ is Schur-positive
exactly when $1\in J$ and, for even $n$, $n\in J$ implies $n/2\in J$.
The same character estimates prove the Ramanujan-square conjecture of
Shareshian and Sundaram: the function
$\sum_{d\mid n}c_d(n/d)^2p_d^{n/d}$, where $c_d(r)$ is the Ramanujan sum,
has a positive coefficient of $s_\lambda$ for every $n\ge1$ and
$\lambda\vdash n$, except when $n\equiv2\pmod4$ and $\lambda=(1^n)$,
in which case the coefficient is zero.

We prove that an element of this space has integral Schur coefficients
if and only if its $Q$-coordinates are integral.
The Boolean classification also determines the convex hull of the
Schur-positive Boolean points with $Q_{n,1}$-coordinate $1$.
We compute its Ehrhart polynomial and volume, prove its integer
decomposition property, and determine the Hilbert basis of its cone.
For $n\ge18$, we prove that setting the coefficient of $s_{(n)}$
or $s_{(1^n)}$ equal to $0$ or $1$ defines a facet of the section
of the Schur-positive cone with $Q_{n,1}$-coordinate $1$.
The positivity results and coordinate formulas also yield inequalities
for major-index residue multiplicities.
\end{abstract}
\maketitle

\section{Introduction}\label{sec:introduction}
Characters induced from a subgroup generated by an $n$-cycle are
supported on the cycle types $(d^{n/d})$, $d\mid n$.  Their Frobenius
characteristics span the corresponding power-sum space, but their
nonnegative span need not exhaust its Schur-positive cone.
We give a unified proof of four conjectures of Sundaram by classifying
the Schur-positive Boolean sums defined below; the case of sums over divisors up to a
prescribed bound recovers Hou's theorem.  The same character estimates
prove the Ramanujan-square conjecture of Shareshian and Sundaram.
We also identify the integral character lattice of this space and derive
geometric and representation-theoretic consequences of the positivity results.

Let $\Lambda_{\mathbb Q}$ denote the ring of symmetric functions over
$\mathbb Q$, with homogeneous component $\Lambda_{\mathbb Q}^n$.
We write $p_\lambda$ and $s_\lambda$ for the power-sum and Schur functions,
and $[p_\lambda]F$ and $[s_\lambda]F$ for their coefficients in $F$.
A symmetric function is \emph{Schur-positive} if all its Schur
coefficients are nonnegative real numbers; zero coefficients are allowed.
Define
\begin{equation}\label{eq:intro-cone}
\begin{aligned}
 \Rspace_n&\coloneqq
   \operatorname{span}_{\mathbb Q}\{p_d^{n/d}:d\mid n\},\\
 \Rspace_{n,\mathbb R}&\coloneqq\Rspace_n\otimes_{\mathbb Q}\mathbb R,\\
 \Ccone_n&\coloneqq
   \{F\in\Rspace_{n,\mathbb R}:[s_\lambda]F\ge0
                     \text{ for all }\lambda\vdash n\}.
\end{aligned}
\end{equation}
We call $\Ccone_n$ the \emph{cyclic-induction Schur cone}.
Here and below, the coefficient of $s_{(n)}$ is the \emph{trivial Schur
coefficient}, and the coefficient of $s_{(1^n)}$ is the \emph{sign Schur
coefficient}.  For the Frobenius characteristic of an $S_n$-module,
these are the
multiplicities of the trivial and sign representations, respectively.

Put $C_n\coloneqq\langle\mathsf c\rangle$, where
$\mathsf c\coloneqq(1\,2\,\cdots\,n)$, and let $V_r$ be the
one-dimensional $C_n$-module on which $\mathsf c$ acts by
$\zeta_n^r$, with $\zeta_n\coloneqq\exp(2\pi i/n)$.
Thus $V_{r+n}=V_r$, and $V_n=V_0$.  Write $\ch$ for the
Frobenius characteristic map, and set
\[
 \Lie_n^{(r)}\coloneqq\operatorname{Ind}_{C_n}^{S_n}V_r,
 \qquad \ell_n^{(r)}\coloneqq\ch\Lie_n^{(r)}.
\]
For a finite-dimensional module $M$ with character $\chi_M$, we use $\ch M$ to mean
$\ch(\chi_M)$.  Klyachko identified $\Lie_n^{(1)}$ with the
multilinear component of the free Lie algebra on $n$ generators
\cite{Klyachko}.
The functions $\ell_n^{(r)}$ are the \emph{Foulkes characteristics}.
For $r\mid n$, they form a basis of $\Rspace_n$
\cite[Proposition~17]{ShareshianSundaram}; see
\Cref{key:proposition-cyclic-induction-space} for the corresponding
description in terms of class functions.

Let $\D(n)\coloneqq\{d\ge1:d\mid n\}$, ordered by divisibility.
To study Sundaram's functions, we use the basis
\[
 Q_{n,d}\coloneqq\ell_{n/d}^{(1)}[p_d]\qquad(d\mid n),
\]
where brackets denote plethysm.  Sundaram's identities
\cite[Theorem~6.4 and Corollary~6.12]{SundaramVariations} give
\[
 \ell_n^{(r)}=\sum_{d\mid r}Q_{n,d}\quad(r\mid n),
 \qquad
 f_n^T\coloneqq\sum_{d\in T\cap\D(n)}Q_{n,d}
       \quad(T\subseteq\mathbb Z_{\ge1}).
\]
The functions $f_n^T$ were introduced in
\cite[Definitions~6.1--6.2]{SundaramVariations}.
M\"obius inversion shows that the $Q_{n,d}$ also form a basis of
$\Rspace_n$ and have integral Schur coefficients.
For $F=\sum_{d\mid n}x_d(F)Q_{n,d}$, we call the $x_d(F)$ its
\emph{$Q$-coordinates}.  A vector is \emph{Boolean} if every entry is
$0$ or $1$.  Accordingly, a \emph{Boolean $Q$-sum} is
\[
 Q_{n,J}\coloneqq\sum_{d\in J}Q_{n,d}\qquad(J\subseteq\D(n));
\]
we call $J$ its \emph{$Q$-support}.

Sundaram conjectured Schur positivity for the families $f_n^T$ with
\[
 T=\{1,k\},\qquad T=\{k^j:j\ge0\},\qquad T=\{1,2,\ldots,k\}.
\]
The first conjecture concerns $k=2$ and odd $k\ge3$, the second odd
$k\ge3$, and the third all $k\ge1$.
She also conjectured Schur positivity of
$\prod_{j\ge0}(1-p_{k^j})^{-1}$ for odd $k\ge3$.
These conjectures first appeared in
\cite[Conjectures~4--7]{SundaramVariations} and were subsequently stated
in \cite[Conjectures~1--4]{SundaramPrimePower}.
The product is understood degree by degree in symmetric functions.

In \Cref{thm:boolean-classification}, we prove that, for every
$J\subseteq\D(n)$, the sum $Q_{n,J}$ is Schur-positive if and only if
$J=\varnothing$, or
\begin{equation}\label{eq:intro-boolean-condition}
 1\in J,
 \qquad
 2\mid n\ \Longrightarrow\
 (n\in J\Longrightarrow n/2\in J).
\end{equation}
If $1\in J$ but the second condition fails, the coefficient of
$s_{(1^n)}$ is $-1$ and all other Schur coefficients are nonnegative.
Consequently, $f_n^T$ is Schur-positive for every $n$ precisely when
$T=\varnothing$, or $1\in T$ and $2m\in T$ implies $m\in T$
for every $m\ge1$ (\Cref{cor:global-boolean}).
This condition proves Sundaram's conjectures for
$T=\{1,k\}$ and $T=\{k^j:j\ge0\}$ and gives the complete classifications,
including the exceptional even values of $k$
(\Cref{cor:one-k-classification,cor:powers-classification}).
It also recovers Schur positivity for $T=\{1,2,\ldots,k\}$ for all
$n,k\ge1$, as proved by Hou \cite[Theorem~1.1]{Hou}.
Sundaram's plethystic identity for
$\prod_{d\in T}(1-p_d)^{-1}$ then gives the product conjecture
(\Cref{cor:product-implication}).

The second positivity problem lies in the same power-sum space but is
not restricted to Boolean $Q$-sums.  It is defined in terms of Ramanujan sums.
For $d\ge1$ and $r\in\mathbb Z$, let
\[
 c_d(r)\coloneqq
 \sum_{\substack{1\le j\le d\\\gcd(j,d)=1}}\exp(2\pi i jr/d).
\]
Foulkes' formula expresses $\ell_n^{(r)}$ in terms of these integers
\cite[Theorem~1]{Foulkes}.  Shareshian and Sundaram considered
$\sum_{d\mid n}c_d(n/d)^u p_d^{n/d}$ for nonnegative integers $u$,
with every weight understood to be $1$ when $u=0$.
They proved Schur positivity for $u=0,1$ by expressing the functions
as nonnegative integral combinations of Foulkes characteristics
\cite[Theorem~36(1)]{ShareshianSundaram}, and conjectured
Schur positivity for $u=2$ \cite[Conjecture~38]{ShareshianSundaram}.
We prove this conjecture in \Cref{thm:ramanujan-square}.  In fact,
\[
 R_n^{\sq}\coloneqq\sum_{d\mid n}c_d(n/d)^2p_d^{n/d}
\]
satisfies, for every $n\ge1$ and $\lambda\vdash n$,
\[
 [s_\lambda]R_n^{\sq}\ge0,
 \qquad
 [s_\lambda]R_n^{\sq}=0
 \ \Longleftrightarrow\
 \bigl(n\equiv2\pmod4\ \text{and}\ \lambda=(1^n)\bigr).
\]

Both positivity arguments use the expression
\[
 F_n(w)\coloneqq\frac1n\sum_{d\mid n}w(d)p_d^{n/d},
 \qquad w(1)=1,
\]
and the coefficient formula
\[
 n[s_\lambda]F_n(w)
 =f^\lambda+
   \sum_{\substack{d\mid n\\d>1}}w(d)\chi^\lambda_{(d^{n/d})}.
\]
Here $\chi^\lambda$ is the irreducible character indexed by $\lambda$,
$\chi^\lambda_{(d^{n/d})}$ is its value on that cycle type, and
$f^\lambda\coloneqq\chi^\lambda(1)$.
This applies to every Boolean sum containing $Q_{n,1}$ and to
$R_n^{\sq}/n$.  For $n\ge18$, the sum over $d>1$ has absolute value less
than $f^\lambda$, except at the four partitions
$(n),(n-1,1),(2,1^{n-2})$, and $(1^n)$, whose coefficients are explicit.
Following Hou \cite[Sections~5--6]{Hou}, we use the Bernstein identity
and a hook-length comparison for partitions with a first row or column
longer than $n/2$, and Swanson's dimension estimate and the
Fomin--Lulov character bound for the remaining partitions.
The extension to arbitrary Boolean supports and to Ramanujan-square
weights requires estimates beyond a uniform bound $|w(d)|\le\kappa d$:
neither family admits such a constant $\kappa$
(\Cref{rem:boolean-not-linear,rem:ramanujan-not-linear}).
\Cref{thm:criterion} gives a common criterion; the weight estimates in
\Cref{sec:boolean,sec:ramanujan} verify its hypotheses, and
\Cref{app:computations} treats the cases $n<18$ using individual character
estimates, known nonnegative Foulkes expansions, and explicit character
calculations.

To relate Schur positivity to representations, we determine the integral
lattice in $\Rspace_{n,\mathbb R}$.
With
$\Lambda_{\mathbb Z}^n\coloneqq
\bigoplus_{\lambda\vdash n}\mathbb Zs_\lambda$,
\Cref{thm:Q-lattice-saturation} proves
\begin{equation}\label{eq:intro-lattice-saturation}
 \Rspace_{n,\mathbb R}\cap\Lambda_{\mathbb Z}^n
 =\bigoplus_{d\mid n}\mathbb ZQ_{n,d}
 =\bigoplus_{d\mid n}\mathbb Z\ell_n^{(d)}.
\end{equation}
The proof recovers the coordinates successively from selected two-row
Schur coefficients and, in even degree, the coefficient of $s_{(1^n)}$.
Nonnegative integral Foulkes coordinates describe modules induced
from $C_n$, whereas integral Schur-positive functions can have negative
Foulkes coordinates.  For a Boolean sum, nonnegative Foulkes coordinates
occur exactly when $J=\varnothing$ or
$J=\D(r)\coloneqq\{d\in\D(n):d\mid r\}$ for some $r\mid n$
(\Cref{prop:fixed-foulkes}).

The distinction between Schur positivity and nonnegative Foulkes
coordinates has a geometric form.
Normalize by $x_1=1$ and put
\[
 \begin{aligned}
 P_n^Q&\coloneqq
 \{x\in\mathbb R^{\D(n)}:x_1=1,\
        \sum_{d\mid n}x_dQ_{n,d}\in\Ccone_n\},\\
 \Delta_n&\coloneqq\conv\{\one_{\D(r)}:r\mid n\},
 \end{aligned}
\]
where $\conv$ denotes convex hull and $\one_A$ denotes the indicator
vector of $A$.  Then $\Delta_n$ consists of the $Q$-coordinate vectors of
normalized nonnegative combinations of Foulkes characteristics.
If $H_n$ is the convex hull of
the Boolean points of $P_n^Q$, the classification gives
\begin{equation}\label{eq:intro-nested-regions}
 \Delta_n\subseteq H_n\subseteq P_n^Q,
 \qquad H_n=P_n^Q\cap[0,1]^{\D(n)}.
\end{equation}
The polytope $H_n$ is a cube in odd degree, an interval for $n=2$,
and a product of a cube and a triangle in even degree at least four
(\Cref{thm:boolean-core}).  This description gives its Ehrhart polynomial,
volume, integer decomposition property, and the Hilbert basis of its cone.
For these terms and the volume convention, see
\Cref{subsec:foulkes-simplex,subsec:boolean-polytope} and
\cite{BeckRobins,BrunsGubeladze}.
\Cref{ex:degree-nine-regions} determines these regions in degree $9$:
the $Q$-coordinate vector of $f_9^{\{1,9\}}$ lies in
$H_9\setminus\Delta_9$, whereas that of $R_9^{\sq}/9$
lies in $P_9^Q\setminus H_9$.

The trivial and sign Schur coefficients that enter both positivity
arguments also define boundary facets.  Define
\begin{equation}\label{eq:intro-endpoint-map}
 \pi_n(F)\coloneqq([s_{(n)}]F,[s_{(1^n)}]F),
 \qquad
 \pi_n^Q(x)\coloneqq\pi_n\left(\sum_{d\mid n}x_dQ_{n,d}\right).
\end{equation}
For $n\ge18$, \Cref{thm:endpoint-facets} determines the images of
$\Delta_n,H_n,P_n^Q$ under this map and the facets of $P_n^Q$ obtained by setting
either coefficient equal to $0$ or $1$.  Thus the same coefficients
govern the Boolean classification, the exceptional zeros of
$R_n^{\sq}$, and these boundary faces.

For $n\ge2$, an integral Schur-positive $F$ with $x_1(F)=1$ is the
Frobenius characteristic of an $S_n$-module of dimension $(n-1)!$
whose restriction to $S_{n-1}$ is regular (\Cref{prop:branching}).
The Kra\'skiewicz--Weyman theorem \cite[Theorem~1]{KW} identifies the
Schur coefficients of $\ell_n^{(r)}$ with major-index residue
multiplicities of standard Young tableaux.  Applying this identification
to the $Q$-expansions gives the inequalities in
\Cref{sec:lifts-residue}, including the equality cases for
$Q$-supports $\{1,k\}$ with composite $k\mid n$ and $n>k$.

The results are organized as follows.
\Cref{sec:coordinates} establishes the coordinate changes and integral
lattice, and \Cref{sec:criterion} proves the common character criterion.
\Cref{sec:boolean} classifies the Schur-positive Boolean sums and deduces
Sundaram's four conjectures.
\Cref{sec:boolean-geometry} determines the geometry and integral
decompositions of $H_n$.
\Cref{sec:ramanujan} proves the Ramanujan-square conjecture.
\Cref{sec:boundary} then compares the boundary faces of the normalized
regions, and \Cref{sec:lifts-residue} develops the module and tableau
consequences.
\Cref{sec:further-questions} lists three open problems.
The appendices supply the estimates and finite calculations required
by the positivity proofs.

\section{Coordinate systems, cyclic lifts, and the integral character lattice}
\label{sec:coordinates}

Power-sum coordinates determine class-function values, while Foulkes
coordinates specify a linear combination of characters induced from $C_n$.
The Boolean sums in Sundaram's conjectures have $Q$-coordinates equal
to $0$ or $1$.  We use these coordinate descriptions to distinguish
Schur positivity from nonnegative Foulkes coordinates, identify the
integral character lattice, and express Schur coefficients in terms of
major-index residue multiplicities.

\subsection{Foulkes characteristics, the \texorpdfstring{$Q$}{Q}-basis, and major index}
\label{subsec:foulkes-bases}

All representations are over $\mathbb C$.  We use the conventions of
\cite[Chapter~I]{Macdonald} for symmetric functions, the Hall inner
product, the Frobenius characteristic, and plethysm.
Put $\mathbb N\coloneqq\set{1,2,\ldots}$,
$[m]\coloneqq\set{1,\ldots,m}$ for $m\ge1$, $[0]\coloneqq\varnothing$,
and $\zeta_d\coloneqq\exp(2\pi i/d)$ for $d\ge1$.
The symbols $\mu$ and $\phi$ denote the number-theoretic M\"obius and
Euler totient functions, and $\rad(m)\coloneqq\prod_{p\mid m}p$ is the
squarefree radical of $m\ge1$.
For a statement $P$, let $\one_P$ be $1$ if $P$ holds and $0$ otherwise;
for $A\subseteq\D(n)$, put $\one_A\coloneqq(\one_{d\in A})_{d\mid n}$.

We retain the notation $\Lambda_{\mathbb Q}$, $\Lambda_{\mathbb Q}^n$,
$p_\lambda$, and $s_\lambda$ from the introduction, and write
$\Lambda_{\mathbb Z}^n\coloneqq\bigoplus_{\lambda\vdash n}\mathbb Zs_\lambda$
for the degree-$n$ integral Schur lattice.
Write $h_r,e_r$ for the complete homogeneous and elementary symmetric
functions.  For $\lambda\vdash n$, let $\lambda'$ be its conjugate,
$S^\lambda$ the Specht module, $\chi^\lambda$ its character, and
$f^\lambda\coloneqq\dim S^\lambda$; $\chi^\lambda_\rho$ denotes the value
on permutations of cycle type $\rho$.
With $\CF(S_n)$ denoting the complex vector space of class functions
on $S_n$, write
\[
 \ch:\CF(S_n)\longrightarrow
 \Lambda_{\mathbb Q}^n\otimes_{\mathbb Q}\mathbb C.
\]
For a finite-dimensional $S_n$-module $M$, put $\ch M\coloneqq\ch(\chi_M)$ and extend
this notation $\mathbb Z$-linearly to virtual modules.
For $F,G$ of the same degree, write $F\ge_s G$ if $F-G$ is Schur-positive;
for formal series, interpret this coefficientwise.

Recall that $\Lie_n^{(r)}=\operatorname{Ind}_{C_n}^{S_n}V_r$ and
$\ell_n^{(r)}=\ch\Lie_n^{(r)}$, with $V_r$ defined in the introduction.
The character of $\Lie_n^{(r)}$ is a Foulkes character, whereas
$\ell_n^{(r)}$ denotes its Frobenius characteristic.
For the Ramanujan sums $c_d(r)$ defined in the introduction, we use
the standard formulas
\begin{equation}\label{eq:ramanujan-standard}
 c_d(r)=\sum_{e\mid\gcd(d,r)}e\mu(d/e)
 =\frac{\phi(d)}{\phi(d/\gcd(d,r))}
  \mu\left(\frac d{\gcd(d,r)}\right).
\end{equation}
The divisor-sum formula follows by M\"obius inversion; the formula on the
right is due to Kluyver \cite[p.~410]{Kluyver}.  See also
\cite[Section~2.1]{ShareshianSundaram}.
Foulkes' formula \cite[Theorem~1]{Foulkes} is
\begin{equation}\label{eq:foulkes}
 \ell_n^{(r)}=\frac1n\sum_{d\mid n}c_d(r)p_d^{n/d}.
\end{equation}
It gives the classification
\[
 \Lie_n^{(r)}\cong\Lie_n^{(r')}
 \quad\Longleftrightarrow\quad
 \gcd(n,r)=\gcd(n,r')
 \qquad(1\le r,r'\le n).
\]
The forward implication follows by comparing coefficients and using
\[
 \sum_{e\mid d}c_e(r)=d\one_{d\mid r};
\]
the converse follows from
\eqref{eq:ramanujan-standard}.

For $T\subseteq\mathbb N$, the functions $f_n^T$ introduced above have
the power-sum expansion
\begin{equation}\label{eq:psi-fT}
 \psi^T(d)\coloneqq\sum_{\substack{e\mid d\\e\in T}}e\mu(d/e),
 \qquad
 f_n^T=\frac1n\sum_{d\mid n}\psi^T(d)p_d^{n/d}.
\end{equation}
For a nonempty set $T$, the functions $f_n^T$ were introduced in
\cite[Definitions~6.1--6.2]{SundaramVariations}; see also
\cite[Definitions~4.1--4.2]{SundaramPrimePower}.
For the empty set, $f_n^\varnothing=0$.
Recall that $Q_{n,d}=\ell_{n/d}^{(1)}[p_d]$ for $d\mid n$.
Since the $\ell_n^{(e)}$, $e\mid n$, form a basis of $\Rspace_n$
\cite[Proposition~17]{ShareshianSundaram}, Sundaram's identities below
give the $Q$-basis by M\"obius inversion.

\begin{proposition}[{\cite[Theorem~6.4 and Corollary~6.12]{SundaramVariations}}]
\label{prop:Q-basis}
The functions $Q_{n,d}$, where $d\mid n$, form a basis of $\Rspace_n$.
The change-of-basis formulas are
\begin{align}
 \ell_n^{(r)}&=\sum_{d\mid r}Q_{n,d} &&(r\mid n),
 \label{eq:zeta-Q-Foulkes}\\
 Q_{n,d}&=\sum_{e\mid d}\mu(d/e)\ell_n^{(e)} &&(d\mid n).
 \label{eq:mobius-Q-Foulkes}
\end{align}
Moreover, for every $T\subseteq\mathbb N$,
\begin{equation}\label{eq:fT-Q-sum}
 f_n^T=\sum_{d\in T\cap\D(n)}Q_{n,d}.
\end{equation}
Consequently, every $f_n^T$ has integral Schur coefficients.
\end{proposition}

Formula~\eqref{eq:fT-Q-sum} is recorded in
\cite[equation~(6.2)]{SundaramVariations}, and
\eqref{eq:zeta-Q-Foulkes} is its specialization in
\cite[Corollary~6.12]{SundaramVariations}.  The inverse formula
\eqref{eq:mobius-Q-Foulkes} follows by M\"obius inversion.

\begin{proposition}\label{key:proposition-cyclic-induction-space}
The Frobenius characteristic map identifies each of the following subspaces
of $\CF(S_n)$ with
$\Rspace_n\otimes_{\mathbb Q}\mathbb C$:
\begin{enumerate}[label=\textup{(\roman*)}]
\item the image of
$\operatorname{Ind}_{C_n}^{S_n}:\CF(C_n)\to\CF(S_n)$, and
\item the space of class functions on $S_n$ supported on conjugacy classes
that meet $C_n$.
\end{enumerate}
\end{proposition}

\begin{proof}
Induction gives the inclusion \textup{(i)}$\subseteq$\textup{(ii)}.
The classes meeting $C_n$ have types $(d^{n/d})$, so \textup{(ii)} has
dimension $|\D(n)|$.  The induced characters from $V_e$, $e\mid n$,
have linearly independent characteristics spanning $\Rspace_n$
\cite[Proposition~17]{ShareshianSundaram}.  Hence both spaces have the
stated Frobenius image.
\end{proof}

Foulkes' formula and $c_u(1)=\mu(u)$ give
\begin{equation}\label{eq:Q-rectangular-expansion}
 Q_{n,d}=\frac dn\sum_{u\mid n/d}\mu(u)p_{du}^{n/(du)}.
\end{equation}

We use the English convention for Young diagrams.  For $\lambda\vdash n$,
let $\operatorname{SYT}(\lambda)$ be the set of standard Young tableaux of
shape $\lambda$.  For $U\in\operatorname{SYT}(\lambda)$, set
\[
 \begin{aligned}
  \operatorname{Des}(U)
    &\coloneqq\set{i\in[n-1]:i+1\text{ lies in a lower row than }i},\\
  \maj(U)&\coloneqq\sum_{i\in\operatorname{Des}(U)}i.
 \end{aligned}
\]
For $r\in\mathbb Z$, define
\[
 \begin{aligned}
 \operatorname{SYT}_r(\lambda)
 &\coloneqq\set{U\in\operatorname{SYT}(\lambda):
                  \maj(U)\equiv r\pmod n},\\
 a_{\lambda,r}&\coloneqq\#\operatorname{SYT}_r(\lambda).
 \end{aligned}
\]
The Kra\'skiewicz--Weyman theorem \cite[Theorem~1]{KW} gives
\begin{equation}\label{eq:KW}
 [s_\lambda]\ell_n^{(r)}=a_{\lambda,r}.
\end{equation}
The second subscript of $a_{\lambda,r}$ is always read modulo $n$.  In
particular, $a_{\lambda,n}=a_{\lambda,0}$.

For $d\mid n$, put
\begin{equation}\label{eq:gamma-major-index-definition}
 \gamma_{\lambda,d}\coloneqq[s_\lambda]Q_{n,d}
 =\sum_{e\mid d}\mu(d/e)a_{\lambda,e}.
\end{equation}

For $x=(x_d)_{d\mid n}\in\mathbb R^{\D(n)}$, define
\begin{equation}\label{eq:weighted-Q-and-beta}
 \mathcal Q_n(x)\coloneqq\sum_{d\mid n}x_dQ_{n,d},
 \qquad
 \beta_x(e)\coloneqq\sum_{\substack{d\mid n\\e\mid d}}\mu(d/e)x_d
 \quad(e\mid n).
\end{equation}
Also define
\begin{equation}\label{eq:weighted-beta-parts}
 \beta_x^+(e)\coloneqq\max\{\beta_x(e),0\},
 \qquad
 \beta_x^-(e)\coloneqq\max\{-\beta_x(e),0\}.
\end{equation}
M\"obius inversion gives the following change-of-basis formulas.

\begin{proposition}
\label{prop:weighted-major-index-dictionary}
For every $x\in\mathbb R^{\D(n)}$,
\begin{align}
 \mathcal Q_n(x)&=\sum_{e\mid n}\beta_x(e)\ell_n^{(e)},
 \label{eq:weighted-Foulkes-expansion}\\
 x_d&=\sum_{\substack{e\mid n\\d\mid e}}\beta_x(e)
 \qquad(d\mid n),
 \label{eq:weighted-beta-inverse}\\
 [s_\lambda]\mathcal Q_n(x)
 &=\sum_{e\mid n}\beta_x(e)a_{\lambda,e}
  =\sum_{d\mid n}x_d\gamma_{\lambda,d}
 \qquad(\lambda\vdash n).
 \label{eq:weighted-major-index-coefficient}
\end{align}
Moreover, $x\in\mathbb Z^{\D(n)}$ if and only if
$\beta_x(e)\in\mathbb Z$ for every $e\mid n$.  The function
$\mathcal Q_n(x)$ is Schur-positive if and only if
\begin{equation}\label{eq:weighted-residue-inequality}
 \sum_{e\mid n}\beta_x^-(e)a_{\lambda,e}
 \le \sum_{e\mid n}\beta_x^+(e)a_{\lambda,e}
 \qquad(\lambda\vdash n).
\end{equation}
\end{proposition}

\begin{proof}
Expand each $Q_{n,d}$ by \eqref{eq:mobius-Q-Foulkes}, apply M\"obius
inversion, and then take Schur coefficients using \eqref{eq:KW}.
Both coordinate transformations have integral entries, proving the
integrality assertion.  Separating the positive and negative parts of
$\beta_x$ gives \eqref{eq:weighted-residue-inequality}.
\end{proof}

For the Boolean sum $Q_{n,J}=\sum_{d\in J}Q_{n,d}$, taking $x=\one_J$ in
\eqref{eq:weighted-major-index-coefficient} gives
\begin{equation}\label{eq:boolean-major-index}
 [s_\lambda]Q_{n,J}
 =\sum_{d\in J}\sum_{e\mid d}\mu(d/e)a_{\lambda,e}.
\end{equation}
\subsection{Schur coefficients and the normalized slice}\label{subsec:normalized-cone}

A \emph{divisor weight} is a function $w:\D(n)\to\mathbb C$, with associated
symmetric function
\begin{equation}\label{eq:Fnw}
 F_n(w)\coloneqq\frac1n\sum_{d\mid n}w(d)p_d^{n/d}.
\end{equation}
For every partition $\lambda\vdash n$,
\begin{equation}\label{eq:general-schur-coefficient}
 n[s_\lambda]F_n(w)
 =w(1)f^\lambda+
 \sum_{\substack{d\mid n\\d>1}}w(d)\chi^\lambda_{(d^{n/d})}.
\end{equation}
When $w(1)=1$, the identity-class term is $f^\lambda$.
To prove positivity, we estimate the sum of the terms with $d>1$.

For a real divisor weight $w:\D(n)\to\mathbb R$, define
\begin{equation}\label{eq:tau-epsilon}
 \tau_n(w)\coloneqq\frac1n\sum_{d\mid n}w(d),
 \qquad
 \eps_n(w)\coloneqq\frac1n\sum_{d\mid n}(-1)^{n-n/d}w(d).
\end{equation}
Via the linear isomorphism $w\mapsto F_n(w)$, we also regard $\tau_n$ and
$\eps_n$ as linear functionals on $\Rspace_{n,\mathbb R}$.  Explicitly,
\[
 \tau_n(w)=[s_{(n)}]F_n(w),
 \qquad
 \eps_n(w)=[s_{(1^n)}]F_n(w).
\]
For $F_n(w)=\ch M$, these are the multiplicities of the trivial and
sign representations in $M$.

For $n\ge4$, we call
$(n),(n-1,1),(2,1^{n-2})$, and $(1^n)$ the four \emph{distinguished
partitions}.  They index, respectively, the trivial, standard,
sign-twisted standard, and sign representations.

\begin{lemma}
\label{lem:endpoint-coefficients}\label{eq:endpoint-explicit}
Let $n\ge4$, and let $w:\D(n)\to\mathbb R$ satisfy $w(1)=1$.  Then
\begin{align*}
 [s_{(n)}]F_n(w)&=\tau_n(w),&
 [s_{(n-1,1)}]F_n(w)&=1-\tau_n(w),\\
 [s_{(2,1^{n-2})}]F_n(w)&=1-\eps_n(w),&
 [s_{(1^n)}]F_n(w)&=\eps_n(w).
\end{align*}
\end{lemma}

\begin{proof}
The trivial and sign coefficients follow from \eqref{eq:tau-epsilon}.
On a class of type $(d^{n/d})$ with $d>1$, the standard character is
$-1$, whereas its value at the identity is $n-1$.  Thus
\eqref{eq:general-schur-coefficient} gives the coefficient
$1-\tau_n(w)$.  Tensoring with the sign character gives
$1-\eps_n(w)$ for the sign-twisted standard representation.
\end{proof}

For a set $J\subseteq\D(n)$, we write $\psi^J$ for the weight in
\eqref{eq:psi-fT}, viewing $J$ as a subset of $\mathbb N$.  Then
$Q_{n,J}=F_n(\psi^J)$.

\begin{lemma}
\label{lem:boolean-endpoints}
Suppose that $1\in J\subseteq\D(n)$.  Then
\begin{equation}\label{eq:boolean-endpoints}
 \tau_n(\psi^J)=\one_{n\in J},
 \qquad
 \eps_n(\psi^J)=(-1)^{n-1}\one_{n\in J}
 +\one_{2\mid n,\,n/2\in J}.
\end{equation}
\end{lemma}

\begin{proof}
By \eqref{eq:Q-rectangular-expansion} and
$\sum_{u\mid m}\mu(u)=\one_{m=1}$, the trivial Schur coefficient of
$Q_{n,d}$ is $\one_{d=n}$.  If $n$ is odd, every class
$(a^{n/a})$ has sign $1$, so the sign coefficient is the same.
If $n$ is even, put $m=n/d$.  The sign coefficient of $Q_{n,d}$ is
\[
 \frac1m\sum_{u\mid m}\mu(u)(-1)^{m/u}
 =-\one_{m=1}+\one_{m=2},
\]
where the equality follows from $(-1)^v=2\one_{2\mid v}-1$.
Summing these coefficients over $d\in J$ proves the formulas.
\end{proof}

For a cone $K\subseteq\mathbb R^{\D(n)}$, define its dual using the
standard inner product $\inner{\cdot}{\cdot}$ by
\[
 K^\vee\coloneqq\set{y\in\mathbb R^{\D(n)}:
 \inner{y}{w}\ge0\text{ for every }w\in K}.
\]
For $\lambda\vdash n$, put
$v_\lambda\coloneqq(\chi^\lambda_{(d^{n/d})})_{d\mid n}\in\mathbb Z^{\D(n)}$.
The isomorphism $w\mapsto F_n(w)$ then gives
$n[s_\lambda]F_n(w)=\inner{v_\lambda}{w}$.

A cone $K\subseteq\mathbb R^N$ is \emph{full-dimensional} if
$\operatorname{span}_{\mathbb R}K=\mathbb R^N$, \emph{pointed} if
$K\cap(-K)=\{0\}$, and \emph{rational polyhedral} if it is the
intersection of finitely many half-spaces
$\{x:\langle a_i,x\rangle\ge0\}$ with $a_i\in\mathbb Q^N$.
Write $\operatorname{cone}(S)$ for the nonnegative real span of $S$.

For a nonempty convex set $C\subseteq\mathbb R^N$, let $\operatorname{aff}(C)$ be the
smallest affine subspace containing $C$; its interior there is the
\emph{relative interior} $\operatorname{relint}(C)$.
If a nonzero linear functional $\varphi$ satisfies $\varphi\ge\alpha$ on
$C$, then $H=\{x:\varphi(x)=\alpha\}$ is a \emph{supporting hyperplane}
whenever $C\cap H\ne\varnothing$, and $C\cap H$ is an \emph{exposed face}.
A \emph{facet} is an exposed face of codimension one in
$\operatorname{aff}(C)$.  A \emph{rational polytope} is the convex hull
of finitely many rational points.

A \emph{base} of a pointed cone $K$ is a convex subset $B\subseteq K$
such that every $0\ne x\in K$ has a unique multiple $tx\in B$ with $t>0$.
For a linear functional $\varphi$ positive on $K\setminus\{0\}$,
$B_\varphi\coloneqq\{x\in K:\varphi(x)=1\}$ is such a base, called the
\emph{normalized slice with respect to $\varphi$}.

Define the cone of divisor weights corresponding to $\Ccone_n$ by
\begin{equation}\label{eq:weight-cone-definition}
 \Kcone_n\coloneqq
 \set{w\in\mathbb R^{\D(n)}:F_n(w)\in\Ccone_n},
\end{equation}
and define its normalized slice in symmetric-function coordinates by
\begin{equation}\label{eq:Pn-definition}
 P_n\coloneqq\set{F_n(w):w\in\Kcone_n,\ w(1)=1}.
\end{equation}

\begin{proposition}\label{prop:weight-cone}
Equation~\eqref{eq:general-schur-coefficient} identifies the weight cone as
\begin{equation}\label{eq:weight-cone}
 \Kcone_n
 =\set{w\in\mathbb R^{\D(n)}:\inner{v_\lambda}{w}\ge0
 \text{ for every }\lambda\vdash n}.
\end{equation}
The cone $\Kcone_n$ is full-dimensional, pointed, and rational polyhedral,
and
\[
\Kcone_n^\vee=\operatorname{cone}\set{v_\lambda:\lambda\vdash n}.
\]
Moreover,
\begin{equation}\label{eq:normalized-base}
 P_n=\Ccone_n\cap\set{F:[p_1^n]F=1/n},
\end{equation}
and $P_n$ is the normalized slice of $\Ccone_n$ with respect to the functional
$F\mapsto n[p_1^n]F$.  It is a compact rational polytope of dimension
$|\D(n)|-1$.
\end{proposition}

\begin{proof}
Equation~\eqref{eq:general-schur-coefficient} gives
\eqref{eq:weight-cone}, proving rational polyhedrality; the dual formula
follows from polyhedral duality.  The weight supported at $1$, with
value $1$, gives
\[
 F_n(w)=\frac1n p_1^n
 =\frac1n\sum_{\lambda\vdash n}f^\lambda s_\lambda,
\]
whose Schur coefficients $f^\lambda/n$ are all positive.  This interior
point proves full dimensionality.  If $w,-w\in\Kcone_n$, every Schur
coefficient vanishes, so $w=0$; hence the cone is pointed.

For $0\ne F_n(w)=\sum_\lambda m_\lambda s_\lambda\in\Ccone_n$,
evaluating the corresponding class function at the identity gives
$(n-1)!w(1)=\sum_\lambda m_\lambda f^\lambda>0$.
Thus $w(1)=1$ defines a base.  On this slice,
$0\le m_\lambda\le(n-1)!/f^\lambda$, so it is bounded.
The defining rational inequalities make it a compact rational polytope,
and full dimensionality gives its dimension $|\D(n)|-1$.
\end{proof}

The normalization records the value at the identity.  If
$F_n(w)=\ch(\chi)$, then
\begin{equation}\label{eq:weight-dimension}
 w(1)=\frac{\chi(1)}{(n-1)!}.
\end{equation}
The normalized slice also has a branching interpretation.  We write
$\reg_G$ for the regular character of a finite group $G$.

\begin{proposition}\label{prop:branching}
Let $n\ge2$, let $w:\D(n)\to\mathbb R$, and write
\[
 F_n(w)=\sum_{\lambda\vdash n}m_\lambda s_\lambda.
\]
The class function $\chi_w\coloneqq\ch^{-1}(F_n(w))$ satisfies
\[
 \operatorname{Res}_{S_{n-1}}^{S_n}\chi_w
 =w(1)\reg_{S_{n-1}},
\]
where $S_{n-1}$ is the subgroup fixing $n$.  Equivalently, for every
$\mu\vdash n-1$,
\begin{equation}\label{eq:branching-equations}
 \sum_{\substack{\lambda\vdash n\\\lambda/\mu\text{ is one box}}}
 m_\lambda=w(1)f^\mu.
\end{equation}
\end{proposition}

\begin{proof}
Under the Frobenius characteristic map, restriction corresponds to
$p_1^\perp=\partial/\partial p_1$, the adjoint of multiplication by $p_1$
for the Hall inner product.  Equation~\eqref{eq:Fnw} gives
\[
 p_1^\perp F_n(w)=w(1)p_1^{n-1}
 =w(1)\ch\bigl(\reg_{S_{n-1}}\bigr),
\]
which proves the restriction formula.\footnote{The author is grateful to
Sheila Sundaram for communicating the proof given here.  The proposition
may alternatively be proved by applying Mackey's formula.}  Comparing
Schur coefficients by
the Young branching rule \cite[Theorem~2.8.3]{Sagan} gives
\eqref{eq:branching-equations}.
The analogous restriction identity for each Foulkes character is used
in the proof of
\cite[Corollary~19]{ShareshianSundaram}.
\end{proof}

For $n\ge2$ and $F\in P_n$, the class function $\ch^{-1}(F)$ therefore
restricts to $\reg_{S_{n-1}}$.  If $F$ has integral Schur coefficients,
it is the Frobenius characteristic of an $S_n$-module of dimension $(n-1)!$.
We now express the normalization in $Q$-coordinates.

\Needspace{6\baselineskip}
Suppose that $F=F_n(w)$ and write its $Q$-coordinate expansion as
\begin{equation}\label{eq:Q-coordinate-general}
 F=\sum_{d\mid n}x_dQ_{n,d}.
\end{equation}
M\"obius inversion gives
\begin{equation}\label{eq:weight-Q-transform}
 w(d)=\sum_{e\mid d}e\mu(d/e)x_e,
 \qquad
 x_d=\frac1d\sum_{e\mid d}w(e).
\end{equation}
Thus $w(1)=x_1$.  With
$\gamma_\lambda\coloneqq(\gamma_{\lambda,d})_{d\mid n}$ denoting the
normal vector of the Schur inequality indexed by $\lambda$, the slice is
\begin{equation}\label{eq:PnQ}
 P_n^Q
 \coloneqq\set{x\in\mathbb R^{\D(n)}:x_1=1,\ \sum_{d\mid n}\gamma_{\lambda,d}x_d\ge0
 \text{ for every }\lambda\vdash n}.
\end{equation}
For $F=F_n(w)$ as in \eqref{eq:Q-coordinate-general}, the two functionals in
\eqref{eq:tau-epsilon} are
\begin{equation}\label{eq:endpoint-Q}
 \tau_n(w)=x_n,
 \qquad
 \eps_n(w)=
 \begin{cases}
  x_n,&n\text{ odd},\\
  x_{n/2}-x_n,&n\text{ even}.
 \end{cases}
\end{equation}
Consequently, if $F\in P_n$ has $Q$-coordinate vector $x$ and $n\ge4$,
then its four distinguished coefficients are, when $n$ is odd,
\begin{equation}\label{eq:endpoint-odd-Q}
 \begin{aligned}
 [s_{(n)}]F&=x_n,&
 [s_{(n-1,1)}]F&=1-x_n,\\
 [s_{(2,1^{n-2})}]F&=1-x_n,&
 [s_{(1^n)}]F&=x_n,
 \end{aligned}
\end{equation}
and, when $n$ is even,
\begin{equation}\label{eq:endpoint-even-Q}
 \begin{aligned}
 [s_{(n)}]F&=x_n,&
 [s_{(n-1,1)}]F&=1-x_n,\\
 [s_{(2,1^{n-2})}]F&=1-x_{n/2}+x_n,&
 [s_{(1^n)}]F&=x_{n/2}-x_n.
 \end{aligned}
\end{equation}
For Boolean coordinates in even degree, the coefficient of $s_{(1^n)}$ in
\eqref{eq:endpoint-even-Q} is nonnegative exactly when $x_n\le x_{n/2}$.
These formulas also give the supporting hyperplanes considered in
\Cref{sec:boundary}.

\subsection{The Foulkes simplex and the lift map}
\label{subsec:foulkes-simplex}

Order $\D(n)$ by divisibility.  A subset $J\subseteq\D(n)$ is an
\emph{order ideal} if $d\in J$ and $e\mid d$ imply $e\in J$.
For $r\mid n$, the \emph{principal order ideal} generated by $r$ is
$\D(r)\coloneqq\set{d\in\D(n):d\mid r}$.
Its indicator vector is the $Q$-coordinate vector of $\ell_n^{(r)}$
by \eqref{eq:zeta-Q-Foulkes}.  Define
\begin{equation}\label{eq:Delta}
 \Delta_n\coloneqq\conv\set{\one_{\D(r)}:r\mid n}.
\end{equation}

A \emph{lattice polytope} $P\subseteq\mathbb R^N$ is the convex hull of
finitely many integer points.  Its \emph{translation lattice} is
$\Lambda_P\coloneqq\set{u-v:u,v\in\operatorname{aff}(P)}\cap\mathbb Z^N$.
A lattice simplex is \emph{unimodular} if its edge vectors from one
vertex form a $\mathbb Z$-basis of this lattice.

\begin{samepage}
Throughout, $\vol$ is Lebesgue measure on $x_1=1$ in the coordinates
$x_d$, $d>1$.
In dimension zero, a point has volume $1$.

\end{samepage}

\begin{proposition}\label{prop:Delta-unimodular}
The polytope $\Delta_n$ is a simplex of dimension $|\D(n)|-1$.  Its
translation lattice is
$\set{x\in\mathbb Z^{\D(n)}:x_1=0}$, and it is unimodular.  Moreover,
\[
 \vol(\Delta_n)=\frac1{(|\D(n)|-1)!}.
\]
\end{proposition}

\begin{proof}
Order the divisors compatibly with divisibility.  After deleting the row
indexed by $1$, the edge vectors $\one_{\D(r)}-\one_{\D(1)}$, $r\ne1$,
form a unitriangular matrix.  Hence they form a $\mathbb Z$-basis of
$\set{x\in\mathbb Z^{\D(n)}:x_1=0}$, proving unimodularity and the
volume formula.
\end{proof}

Let $R(C_n)$ be the Grothendieck group of finite-dimensional $C_n$-modules
and $[U]$ the class of $U$.  Extend
$\ch\circ\operatorname{Ind}_{C_n}^{S_n}$ $\mathbb Z$-linearly to $R(C_n)$.
With $V_n=V_0$, define
\begin{equation}\label{eq:lift-lattices}
 \Vlat_n\coloneqq\bigoplus_{e\mid n}\mathbb Z[V_e]\subseteq R(C_n),
 \qquad
 \Flattice_n\coloneqq\bigoplus_{e\mid n}\mathbb Z \ell_n^{(e)}\subseteq\Rspace_n.
\end{equation}
Define
\[
 \mathcal I_n\coloneqq\ch\circ\operatorname{Ind}_{C_n}^{S_n}:
 \Vlat_n\longrightarrow\Flattice_n,
 \qquad [V_e]\longmapsto\ell_n^{(e)}.
\]

This map sends the basis $\{[V_e]:e\mid n\}$ to
$\{\ell_n^{(e)}:e\mid n\}$, so it is an isomorphism of free abelian groups.

\begin{definition}\label{def:canonical-lift}
Put $\sigma_n\coloneqq\mathcal I_n^{-1}:\Flattice_n\to\Vlat_n$.
For $F\in\Flattice_n$, its \emph{lift determined by the representatives
$\{V_e:e\mid n\}$} is
\begin{equation}\label{eq:canonical-lift}
 \sigma_n\left(\sum_{e\mid n}b_e\ell_n^{(e)}\right)
 =\sum_{e\mid n}b_e[V_e]\in\Vlat_n
 \qquad(b_e\in\mathbb Z).
\end{equation}
\end{definition}

The restriction to $\Vlat_n$ is essential: nonisomorphic $V_r,V_s$
with $\gcd(n,r)=\gcd(n,s)$ have isomorphic induced modules.

\begin{definition}\label{def:effective-class}
A class $\xi\in R(C_n)$ is \emph{effective} if $\xi=[U]$ for some
finite-dimensional $C_n$-module $U$, equivalently if
$\xi=\sum_{r=0}^{n-1}m_r[V_r]$ with $m_r\in\mathbb Z_{\ge0}$.
\end{definition}

\begin{proposition}\label{thm:effective-lifts}
Let $F=\sum_{e\mid n}b_e\ell_n^{(e)}\in\Flattice_n$.  The following
conditions are equivalent:
\begin{enumerate}[label=\textup{(\roman*)}]
\item $F=\ch\operatorname{Ind}_{C_n}^{S_n}U$ for some
      finite-dimensional $C_n$-module $U$;
\item $b_e\ge0$ for every $e\mid n$;
\item $\sigma_n(F)$ is effective.
\end{enumerate}
Consequently,
\begin{equation}\label{eq:effective-lift-cone}
\begin{aligned}
 &\operatorname{cone}\set{\ch\operatorname{Ind}_{C_n}^{S_n}U:
 U\text{ a finite-dimensional }C_n\text{-module}}\\
 &\hspace{35mm}=\operatorname{cone}\set{\ell_n^{(e)}:e\mid n}.
\end{aligned}
\end{equation}
In $Q$-coordinates, the section $x_1=1$ of this cone is $\Delta_n$.
Among Frobenius characteristics of modules induced from $C_n$, those
with $x_1=1$ are precisely $\ell_n^{(e)}$, $e\mid n$, whose
$Q$-coordinate vectors are the vertices of $\Delta_n$.
\end{proposition}

\begin{proof}
If $U\cong\bigoplus_{r=0}^{n-1}V_r^{\oplus m_r}$, then
$\ch\operatorname{Ind}_{C_n}^{S_n}U$ has Foulkes coefficients
\[
 b_e=\sum_{\substack{0\le r<n\\\gcd(n,r)=e}}m_r\ge0.
\]
Here $\gcd(n,0)=n$.  Conversely, $b_e\ge0$ gives
$U=\bigoplus_{e\mid n}V_e^{\oplus b_e}$ with $[U]=\sigma_n(F)$.
This proves the equivalence and \eqref{eq:effective-lift-cone}.
Since $x_1(F)=\sum_{e\mid n}b_e$, the slice $x_1=1$ has
$Q$-coordinate image $\conv\{\one_{\D(e)}:e\mid n\}=\Delta_n$.
Nonnegative integral coefficients $b_e$ summing to one give exactly
its vertices.
\end{proof}

\subsection{Nonnegative Foulkes coordinates}\label{subsec:global-lifts}
For a set $T\subseteq\mathbb N$, let
\[
 \beta_{n,T}(e)\coloneqq
 \sum_{\substack{a\mid n,\ a\in T\\e\mid a}}\mu(a/e),
 \qquad
 \mathcal F_n^+\coloneqq
 \set{\sum_{e\mid n}b_e\ell_n^{(e)}:b_e\ge0}.
\]
By Proposition~\ref{prop:weighted-major-index-dictionary},
$f_n^T=\sum_{e\mid n}\beta_{n,T}(e)\ell_n^{(e)}$.
\begin{proposition}\label{prop:fixed-foulkes}
Let $n\ge1$ and let $T\subseteq\mathbb N$.  The following are equivalent:
\begin{enumerate}[label=\textup{(\roman*)}]
\item $f_n^T\in\mathcal F_n^+$;
\item $T\cap\D(n)$ is empty or equals $\D(r)$ for some
$r\mid n$;
\item $f_n^T=0$ or $f_n^T=\ell_n^{(r)}$ for some $r\mid n$.
\end{enumerate}
\end{proposition}

\begin{proof}
M\"obius inversion gives
\[
 \one_{e\in T}=\sum_{\substack{a\mid n\\e\mid a}}\beta_{n,T}(a)
 \qquad(e\mid n).
\]
In particular, the integral Foulkes coefficients sum to $\one_{1\in T}$.
If they are nonnegative, then either they all vanish or exactly one,
say $\beta_{n,T}(r)$, equals $1$ and all the others vanish.
The displayed formula then makes
$T\cap\D(n)$ empty or equal to $\D(r)$, proving
\textup{(i)}$\Rightarrow$\textup{(ii)}.
Formula~\eqref{eq:zeta-Q-Foulkes} gives
\textup{(ii)}$\Rightarrow$\textup{(iii)}, and
\textup{(iii)}$\Rightarrow$\textup{(i)} is immediate.
\end{proof}

\subsection{The integral character lattice}\label{sec:integral-lattice}

The change of basis identifies $\Flattice_n$ with
$\bigoplus_{d\mid n}\mathbb ZQ_{n,d}$, whose elements have integral Schur
coefficients.  We prove the converse: every $F\in\Rspace_{n,\mathbb R}$
with integral Schur coefficients belongs to $\Flattice_n$, so the lift
$\sigma_n(F)$ of \Cref{def:canonical-lift} is defined.
For $n\ge2$,
the trivial and standard coefficients determine $x_n$ and $x_1$, and the sign
coefficient determines $x_{n/2}$ in even degree.  Two-row partitions
then recover the coordinates with $1<d<n/2$ successively, without division.
For $N\in\mathbb Z_{\ge0}$, we use the convention
$\binom Nt=0$ unless $t$ is an integer with $0\le t\le N$.

\begin{lemma}
\label{lem:triangular-two-row}
Let $n\ge1$ and $d\mid n$ with $1<d<n/2$, and put
$\lambda_d\coloneqq(n-d,d)$.  Then
\begin{equation}\label{eq:triangular-two-row}
 [s_{\lambda_d}]Q_{n,d}=1,
 \qquad
 [s_{\lambda_d}]Q_{n,e}=0
 \quad\text{whenever }e\mid n\text{ and }d<e\le n.
\end{equation}
\end{lemma}

\begin{proof}
For $0\le j\le n$, the function $h_{n-j}h_j$ is the Frobenius characteristic of the
permutation representation on the $j$-subsets of $\{1,\ldots,n\}$.
For $a\mid n$, a subset fixed by a permutation of cycle type $(a^{n/a})$ is a union
of its $a$-cycles, so the value of this permutation character is
$\one_{a\mid j}\binom{n/a}{j/a}$.
The Jacobi--Trudi identity
$s_{(n-d,d)}=h_{n-d}h_d-h_{n-d+1}h_{d-1}$ therefore gives
\[
 \chi^{(n-d,d)}_{(a^{n/a})}
 =\one_{a\mid d}\binom{n/a}{d/a}
  -\one_{a\mid d-1}\binom{n/a}{(d-1)/a}.
\]
If $a>d$ both terms vanish; if $a=d>1$ the value is $n/d$.
In \eqref{eq:Q-rectangular-expansion} with $e\ge d$, all terms have
$a=eu\ge d$. For $e=d$, only $u=1$ contributes, and its contribution
is $(d/n)(n/d)=1$. For $e>d$ every term vanishes.
\end{proof}

\begin{theorem}\label{thm:Q-lattice-saturation}
For every $n\ge1$,
\begin{equation}\label{eq:Q-lattice-saturation}
 \Rspace_{n,\mathbb R}\cap\Lambda_{\mathbb Z}^n
 =\bigoplus_{d\mid n}\mathbb ZQ_{n,d}
 =\Flattice_n.
\end{equation}
Equivalently, for
$F=\sum_{d\mid n}x_dQ_{n,d}\in\Rspace_{n,\mathbb R}$, the following are
equivalent:
\begin{enumerate}[label=\textup{(\roman*)}]
\item every Schur coefficient of $F$ is an integer;
\item $x_d\in\mathbb Z$ for every $d\mid n$;
\item every coefficient of $F$ in the Foulkes basis
      $\set{\ell_n^{(e)}:e\mid n}$ is an integer.
\end{enumerate}
\end{theorem}

\begin{proof}
By \Cref{prop:Q-basis} and the integral change of basis in
\eqref{eq:zeta-Q-Foulkes} and \eqref{eq:mobius-Q-Foulkes},
\[
 \bigoplus_{d\mid n}\mathbb ZQ_{n,d}
 =\Flattice_n
 \subseteq\Rspace_{n,\mathbb R}\cap\Lambda_{\mathbb Z}^n.
\]

For the reverse inclusion, suppose that
$F=\sum_{d\mid n}x_dQ_{n,d}$ has integral Schur coefficients.
If $n=1$, then $Q_{1,1}=s_{(1)}$, so $x_1\in\mathbb Z$.
Assume $n\ge2$.  Formula~\eqref{eq:endpoint-Q} gives
$x_n=[s_{(n)}]F\in\mathbb Z$.
The standard character has value $n-1$ at the identity and $-1$ on
every class $(d^{n/d})$ with $d>1$.  This holds for all $n\ge2$ and gives
\[
 x_1=[s_{(n-1,1)}]F+x_n\in\mathbb Z.
\]
If $n$ is even, \eqref{eq:endpoint-Q} also gives
$x_{n/2}=[s_{(1^n)}]F+x_n\in\mathbb Z$.
These formulas recover all coordinates when $n=2$ or $3$.

For $n\ge4$, the remaining divisors satisfy $1<d<n/2$.
Take them in increasing order.  By \Cref{lem:triangular-two-row},
\[
 x_d=[s_{(n-d,d)}]F-
 \sum_{\substack{e\mid n\\e<d}}x_e[s_{(n-d,d)}]Q_{n,e}.
\]
Each $[s_{(n-d,d)}]Q_{n,e}$ is integral by \Cref{prop:Q-basis},
and the coordinates $x_e$ occurring on the right have already been recovered.
Induction therefore gives $x_d\in\mathbb Z$ for every $d\mid n$.
The integral change of basis gives the equivalence of
\textup{(i)}--\textup{(iii)}.
\end{proof}

By \Cref{prop:weighted-major-index-dictionary,thm:Q-lattice-saturation},
the Foulkes coordinates $(\beta_x(e))_{e\mid n}$ identify the lattice in
\eqref{eq:Q-lattice-saturation} with $\mathbb Z^{\D(n)}$.
Under this identification, Schur positivity is equivalent to
\begin{equation}\label{eq:integral-major-index-inequalities}
 \sum_{e\mid n}\beta_x(e)a_{\lambda,e}\ge0
 \qquad\text{for every }\lambda\vdash n.
\end{equation}
Since $x_1=\sum_{e\mid n}\beta_x(e)$ by
\eqref{eq:weighted-beta-inverse}, fixing $x_1=h$ for
$h\in\mathbb Z_{\ge0}$ amounts to imposing
$\sum_{e\mid n}\beta_x(e)=h$.

\begin{remark}\label{rem:saturation-versus-idp}
The lattice identification in \eqref{eq:Q-lattice-saturation} does not
imply that, for every integer $h\ge1$, every point
of $hP_n^Q\cap\mathbb Z^{\D(n)}$ is a sum of $h$ points of
$P_n^Q\cap\mathbb Z^{\D(n)}$.  When $P_n^Q$ is a lattice polytope, this
decomposition assertion is precisely its integer decomposition property
with respect to $\mathbb Z^{\D(n)}$.  We prove the corresponding assertion
for the smaller lattice polytope $H_n$ in \Cref{cor:boolean-Ehrhart}, but
make no such claim for $P_n^Q$.
\end{remark}

\section{Uniform character estimates}\label{sec:criterion}

For $n\ge18$, the Boolean classification and Ramanujan-square positivity
use the common criterion of \Cref{thm:criterion}; smaller degrees are
treated in \Cref{app:computations}.
We follow the character-estimate strategy of
\cite[Sections~5 and~6]{Hou}, using $F_n(w)$ from \eqref{eq:Fnw}.
We apply the Bernstein creation identity when
$\lambda_1>n/2$ and deduce the result for $\lambda'_1>n/2$ by tensoring
with the sign representation.  When
$\lambda_1,\lambda'_1\le n/2$, we combine Swanson's dimension bound with
the Fomin--Lulov character estimate.  After separating the hook term,
we bound the remaining long-row contribution using $A_{n,r}(w)$ in
\eqref{eq:Anr}.  For $\lambda_1,\lambda'_1\le n/2$, $E_n(w)$ in \eqref{eq:Enw}
bounds the absolute nonidentity contribution divided by $f^\lambda$.

\subsection{Long first rows and columns}\label{subsec:long-row}

For $2\le r<n/2$, put
\begin{equation}\label{eq:grn}
 g_r(n)\coloneqq\binom nr-\binom n{r-1}.
\end{equation}

Let $\mu\vdash r$ with $2\le r<n/2$, and put
$\lambda\coloneqq(n-r,\mu)$ and
$\delta_\mu\coloneqq\one_{\mu=(1^r)}$.
For each divisor $d\mid n$ with $2\le d\le r$, define
\begin{equation}\label{eq:eta-definition}
 \eta_d\coloneqq
 \sum_{q=1}^{\lfloor r/d\rfloor}(-1)^{r-dq}\binom{n/d}{q}
 \sum_{\substack{\nu\subseteq\mu\\
       \mu/\nu\text{ a vertical }(r-dq)\text{-strip}}}
 \chi^\nu_{(d^q)}.
\end{equation}

\begin{lemma}[{\cite[Lemma~5.1 and equation~(24)]{Hou}}]
\label{lem:bernstein-rectangular}
With this notation, for every divisor $d\mid n$ with $d\ge2$,
\begin{equation}\label{eq:bernstein-rectangular}
 \chi^\lambda_{(d^{n/d})}=
 \begin{cases}
  \delta_\mu(-1)^r+\eta_d,&2\le d\le r,\\
  \delta_\mu(-1)^r,&d>r,
\end{cases}
\end{equation}
and, for every divisor $d\mid n$ with $2\le d\le r$,
\begin{equation}\label{eq:eta-bound}
 |\eta_d|\le f^\mu
 \sum_{q=1}^{\lfloor r/d\rfloor}\binom{n/d}{q}.
\end{equation}
\end{lemma}

\begin{lemma}[{\cite[Lemma~5.2]{Hou}}]
\label{lem:long-row-dimension}
Let $\mu\vdash r$ with $2\le r<n/2$, and put
$\lambda\coloneqq(n-r,\mu)$.  Then
\begin{equation}\label{eq:long-row-dimension}
 f^\lambda\ge f^\mu g_r(n).
\end{equation}
\end{lemma}

Let $w:\D(n)\to\mathbb R$ satisfy $w(1)=1$.  Substituting
\eqref{eq:bernstein-rectangular} into
\eqref{eq:general-schur-coefficient} gives, for
$\lambda=(n-r,\mu)$,
\begin{equation}\label{eq:long-row-Fnw}
 n[s_\lambda]F_n(w)
 =f^\lambda+\delta_\mu(-1)^r\bigl(n\tau_n(w)-1\bigr)
 +\sum_{\substack{d\mid n\\2\le d\le r}}w(d)\eta_d.
\end{equation}
For $2\le r<n/2$, define
\begin{equation}\label{eq:Anr}
 A_{n,r}(w)\coloneqq
 \sum_{\substack{d\mid n\\2\le d\le r}}|w(d)|
 \sum_{q=1}^{\lfloor r/d\rfloor}\binom{n/d}{q}.
\end{equation}
By \eqref{eq:eta-bound}, the absolute value of the last sum in
\eqref{eq:long-row-Fnw} is at most $f^\mu A_{n,r}(w)$.

\subsection{Balanced partitions}\label{subsec:balanced}

For the purposes of this paper, a partition $\lambda\vdash n$ is
\emph{balanced} if both its first row and its first column have length at
most $n/2$.  We write
\begin{equation}\label{eq:balanced-partitions}
 \Bpart_n\coloneqq\set{\lambda\vdash n:\lambda_1\le n/2,\ \lambda_1'\le n/2}
\end{equation}
for the set of balanced partitions of $n$.
For $n\ge1$, put
\begin{equation}\label{eq:Lambda-X}
 m_n\coloneqq\left\lfloor\frac{n-1}{2}\right\rfloor,
 \qquad
 \Ldim_n\coloneqq\frac1{m_n+1}\binom n{m_n},
 \qquad
 X_n\coloneqq\frac{\Ldim_n}{\sqrt{2\pi n}}.
\end{equation}

For a box $(i,j)\in\lambda$, the integer $i+j-1$ is its opposite hook
length.  Following Swanson, define
\[
 N(\lambda)\coloneqq n-\max_{(i,j)\in\lambda}(i+j-1)
\]
and
\[
 N^*(\lambda)\coloneqq
 \begin{cases}
  N(\lambda),&2N(\lambda)+1\le n,\\
  \lfloor(n-1)/2\rfloor,&2N(\lambda)+1>n.
 \end{cases}
\]
Swanson's opposite-hook estimate
\cite[Proposition~4.11 and Corollary~4.13]{Swanson} states that
\begin{equation}\label{eq:swanson-opposite-hook}
 f^\lambda\ge\frac1{m+1}\binom nm
 \qquad(0\le m\le N^*(\lambda)).
\end{equation}

\begin{lemma}[{\cite[Lemma~6.1]{Hou}}]
\label{lem:balanced-dimension}
For every $n\ge1$ and every $\lambda\in\Bpart_n$, one has
$f^\lambda\ge\Ldim_n$.
\end{lemma}

\begin{lemma}[{\cite[Lemma~4.1]{Swanson}}]
\label{lem:normalized-FL}
Let $d\ge2$, $s\ge1$, and $\lambda\vdash n=ds$.  Then
\begin{equation}\label{eq:normalized-FL}
 \frac{|\chi^\lambda_{(d^s)}|}{f^\lambda}
 \le d^{-1/2}e^{d/(12n)}
 \left(\frac{\sqrt{2\pi n}}{f^\lambda}\right)^{1-1/d}.
\end{equation}
\end{lemma}

Inequality~\eqref{eq:normalized-FL} is the exponential form of Swanson's
inequality, obtained from the Fomin--Lulov estimate
\cite[Theorem~1.1]{FominLulov} by Stirling's inequalities.  It is also
recorded in \cite[Lemma~4.2]{Hou}.

Fix $n\ge18$, $\lambda\in\Bpart_n$, and a real divisor weight
$w:\D(n)\to\mathbb R$.  The unimodality of the binomial coefficients and
$m_n+1\le(n+1)/2$ give
\[
 \Ldim_n\ge\frac{n(n-1)}{n+1}>\sqrt{2\pi n}.
\]
The partition $\lambda$ is not a hook, so
$\chi^\lambda_{(n)}=0$ by the Murnaghan--Nakayama rule
\cite[Theorem~4.10.2]{Sagan}.  Every proper divisor of $n$ is at most
$n/2$, so $e^{d/(12n)}\le e^{1/24}$ for such divisors.
Since $X_n>1$, we also have
$X_n^{-(1-1/d)}\le X_n^{-2/3}$ when $d\ge3$.
Hence
\Cref{lem:balanced-dimension,lem:normalized-FL} give
\begin{equation}\label{eq:balanced-error}
 \frac1{f^\lambda}
 \left|\sum_{\substack{d\mid n\\d>1}}
 w(d)\chi^\lambda_{(d^{n/d})}\right|
 \le E_n(w),
\end{equation}
where
\begin{equation}\label{eq:Enw}
 E_n(w)\coloneqq e^{1/24}\left(
 X_n^{-1/2}\sum_{\substack{d\mid n\\d=2}}\frac{|w(d)|}{\sqrt d}
 +X_n^{-2/3}
 \sum_{\substack{d\mid n\\3\le d<n}}\frac{|w(d)|}{\sqrt d}
 \right).
\end{equation}
The sum with $d=2$ is empty when $n$ is odd.
\Cref{lem:X-ratio-appendix} gives
\begin{equation}\label{eq:X-ratio}
 \frac{X_{n+1}}{X_n}\ge\frac95\sqrt{\frac n{n+1}}
 \qquad(n\ge18).
\end{equation}

\subsection{The positivity criterion}\label{subsec:positivity-criterion}

Under the following error bounds, Schur positivity is equivalent to
$\tau_n(w),\eps_n(w)\in[0,1]$.  The quantities $A_{n,r}(w)$ and $E_n(w)$
control the long-row or long-column and balanced cases, respectively.

\begin{theorem}\label{thm:criterion}
Let $n\ge18$, and let $w:\D(n)\to\mathbb R$ satisfy $w(1)=1$.  Assume
that
\begin{enumerate}[label=\textup{(\roman*)}]
\item $|\tau_n(w)|\le1$ and $|\eps_n(w)|\le1$;
\item for every $2\le r<n/2$,
\[
 A_{n,r}(w)<\frac35g_r(n);
\]
\item $E_n(w)<1$.
\end{enumerate}
Then $[s_\lambda]F_n(w)>0$ for every $\lambda\vdash n$ other than
$(n),(n-1,1),(2,1^{n-2})$, and $(1^n)$.
The coefficients at these four partitions are
\begin{align*}
 [s_{(n)}]F_n(w)&=\tau_n(w),&
 [s_{(n-1,1)}]F_n(w)&=1-\tau_n(w),\\
 [s_{(2,1^{n-2})}]F_n(w)&=1-\eps_n(w),&
 [s_{(1^n)}]F_n(w)&=\eps_n(w).
\end{align*}
Consequently, $F_n(w)$ is Schur-positive if and only if
$\tau_n(w),\eps_n(w)\in[0,1]$.
\end{theorem}

\begin{proof}
The distinguished coefficients are given by
\Cref{lem:endpoint-coefficients}.  Every other partition either has a first
row or first column longer than $n/2$, or belongs to $\Bpart_n$.

First let $\lambda=(n-r,\mu)$ with $2\le r<n/2$.  By \eqref{eq:eta-bound}, hypothesis~\textup{(ii)}, and
\Cref{lem:long-row-dimension},
\[
 \left|\sum_{\substack{d\mid n\\2\le d\le r}}w(d)\eta_d\right|
 \le f^\mu A_{n,r}(w)
 <\frac35 f^\mu g_r(n)
 \le\frac35 f^\lambda.
\]
If $\mu\ne(1^r)$, then
\[
 n[s_\lambda]F_n(w)>\frac25f^\lambda>0.
\]
If $\mu=(1^r)$, then $f^\lambda=\binom{n-1}{r}$ and
$|n\tau_n(w)-1|\le n+1$.  Hence
\[
 \begin{aligned}
 n[s_\lambda]F_n(w)
 &>\frac25\binom{n-1}{r}-(n+1)\\
 &\ge\frac25\binom{n-1}{2}-(n+1)
 =\frac{n^2-8n-3}{5}>0.
 \end{aligned}
\]

Next suppose that $\lambda$ has a long first column, so that $\lambda'$ has
a long first row.  Define
\[
 \widetilde w(d)\coloneqq(-1)^{n-n/d}w(d).
\]
Multiplication by the sign character gives
$[s_\lambda]F_n(w)=[s_{\lambda'}]F_n(\widetilde w)$.  The absolute-value
hypotheses are unchanged, while
\[
 \tau_n(\widetilde w)=\eps_n(w),
 \qquad
 \eps_n(\widetilde w)=\tau_n(w).
\]
Hence $[s_{\lambda'}]F_n(\widetilde w)>0$, and therefore
$[s_\lambda]F_n(w)>0$.

Finally, let $\lambda\in\Bpart_n$.  Then
\eqref{eq:general-schur-coefficient},
\eqref{eq:balanced-error}, and hypothesis~\textup{(iii)} give
\[
 n[s_\lambda]F_n(w)\ge f^\lambda(1-E_n(w))>0.
\]
Thus all coefficients indexed by other partitions are positive.  The four
distinguished coefficients are nonnegative exactly when
$\tau_n(w),\eps_n(w)\in[0,1]$.
\end{proof}

\begin{remark}\label{rem:comparison-linear-criterion}
There is no constant $\kappa$, independent of $n$, $J$, and $d$,
such that $|\psi^J(d)|\le\kappa d$ for all Boolean sums; see
\Cref{rem:boolean-not-linear}.
The weights $c_d(n/d)^2$ also fail such a bound; see
\Cref{rem:ramanujan-not-linear}.
Conditions \textup{(ii)} and \textup{(iii)} of \Cref{thm:criterion}
are verified for these two families in \Cref{sec:boolean,sec:ramanujan}.
No optimality is claimed for the constants $3/5$ and $1$ or for the
threshold $n\ge18$.
\end{remark}

\subsection{A common long-row bound}
We next verify the long-row hypothesis of \Cref{thm:criterion}
for weights satisfying $|w(2)|\le1$ when $2\mid n$, and
$|w(d)|\le d^2$ for $d\mid n$, $d\ge3$.
The resulting bound applies to both families considered below.

For integers $n\ge18$ and $2\le r<n/2$, define
\begin{equation}\label{eq:Asqnr}
 A_{n,r}^{\sq}
 \coloneqq\sum_{q=1}^{\lfloor r/2\rfloor}
   \binom{\lfloor n/2\rfloor}{q}
 +\sum_{d=3}^rd^2
   \sum_{q=1}^{\lfloor r/d\rfloor}
   \binom{\lfloor n/d\rfloor}{q}.
\end{equation}

\begin{proposition}\label{prop:ramanujan-short-cycle}
For $n\ge18$ and $2\le r<n/2$,
\begin{equation}\label{eq:ramanujan-short-cycle}
 A_{n,r}^{\sq}<\frac13g_r(n).
\end{equation}
Consequently, if $|w(2)|\le1$ when $2\mid n$ and
$|w(d)|\le d^2$ for $d\ge3$, then $A_{n,r}(w)<g_r(n)/3$.
\end{proposition}

\begin{proof}
By \Cref{lem:short-cycle-monotonicity}, it is enough to take
$n_0\coloneqq\max\set{18,2r+1}$.  For $2\le r\le11$, the exact values in
\Cref{tab:quadratic-short} show that
$A^{\sq}_{n_0,r}/g_r(n_0)<1/10$.

Let $r\ge12$.  Apply \Cref{lem:common-long-row-tail} with
$M(2)=1$ and $M(d)=d^2$ for $d\ge3$.  This gives
\[
 \begin{aligned}
 \frac{A_{2r+1,r}^{\sq}}{g_r(2r+1)}
 &<\frac{(r+2)\sqrt r}{2}
 \left(2^{-r}+\left(\sum_{d=3}^rd^2\right)2^{(1-4r)/3}\right)\\
 &<\frac13.
 \end{aligned}
\]
The second inequality follows from \Cref{lem:ramanujan-tail-decrease};
termwise comparison with the weight bounds proves the final assertion.
\end{proof}

\subsection{Linear weights and a uniform relative-error bound}\label{subsec:linear-weights}
The relative-error estimate obtained here will be used in the
quantitative residue inequality of \Cref{cor:composite-equality}.
The common long-row estimate also applies to every divisor weight satisfying
$|w(2)|\le1$ when $2\mid n$ and
$|w(d)|\le\kappa d$ for $d\mid n$, $d\ge3$, with $0<\kappa\le2$.
For balanced partitions we keep the sharper linear-weight estimate,
which also gives a uniform asymptotic statement.
We shall also use the following divisor-sum estimate.
\begin{lemma}[{\cite[Lemma~4.5]{Hou}}]
\label{key:lemma-square-root-divisor-sum}
For every $n\ge1$,
\[
 \sum_{d\mid n}\sqrt d\le3n^{3/4}.
\]
\end{lemma}

Combining \cref{lem:balanced-dimension,lem:normalized-FL,key:lemma-square-root-divisor-sum}
gives an estimate depending only on
$n$ and $\kappa$ for weights satisfying $|w(d)|\le\kappa d$.  The resulting
argument adapts Hou's treatment of partitions satisfying
$\lambda_1,\lambda_1'\le n/2$
\cite[Section~6, especially Lemma~6.3 and Proposition~6.4]{Hou}.

For $\kappa>0$, define
\begin{equation}\label{key:definition-G-bound}
 \mathcal G_n(\kappa)\coloneqq e^{1/24}
 \left(\frac{X_n^{-1/2}}{\sqrt2}
 +3\kappa n^{3/4}X_n^{-2/3}\right).
\end{equation}
For fixed $n$, the function $\mathcal G_n(\kappa)$ is increasing in
$\kappa$.  The constants are not optimized; the applications require only
$\mathcal G_n(\kappa)<1$.

\begin{proposition}
\label{key:proposition-uniform-balanced}
Let $n\ge18$, let $\lambda\in\Bpart_n$, let $w$ be a divisor weight, and
let $\kappa>0$.  Suppose
\[
 |w(2)|\le1\quad(2\mid n),
 \qquad
 |w(d)|\le\kappa d\quad(d\mid n,\ d\ge3).
\]
Then
\begin{equation}\label{key:uniform-balanced-error}
 \frac1{f^\lambda}
 \left|\sum_{\substack{d\mid n\\d>1}}
 w(d)\chi^\lambda_{(d^{n/d})}\right|
 \le\mathcal G_n(\kappa).
\end{equation}
\end{proposition}

\begin{proof}
By \eqref{eq:Enw} and the assumed weight bounds,
\[
 E_n(w)\le e^{1/24}\left(
 \frac{X_n^{-1/2}}{\sqrt2}
 +\kappa X_n^{-2/3}
 \sum_{\substack{d\mid n\\3\le d<n}}\sqrt d
 \right)
 \le\mathcal G_n(\kappa),
\]
where the last inequality follows from
\Cref{key:lemma-square-root-divisor-sum}.
Equation~\eqref{eq:balanced-error} now gives
\eqref{key:uniform-balanced-error}.
\end{proof}

Appendix~\ref{app:estimates} proves that, for fixed $\kappa$,
$\mathcal G_n(\kappa)$ decreases
with $n$ and that $\mathcal G_{18}(2)<1$.  Since
$\mathcal G_n(\kappa)$ is increasing in $\kappa$, it follows that
\[
 \mathcal G_n(\kappa)<1\qquad(n\ge18,\ 0<\kappa\le2).
\]

\Cref{key:proposition-uniform-balanced} also yields the following
relative-error estimate.
\begin{corollary}\label{key:corollary-quantitative-balanced}
Let $n\ge18$, let $\kappa>0$, and let $w$ be a divisor weight with $w(1)=1$
satisfying the weight bounds in
\cref{key:proposition-uniform-balanced}.  Then, for every
$\lambda\in\Bpart_n$,
\begin{equation}\label{key:quantitative-balanced-coefficient}
 \left|
 [s_\lambda]F_n(w)-\frac{f^\lambda}{n}
 \right|
 \le
 \mathcal G_n(\kappa)\frac{f^\lambda}{n}.
\end{equation}
For fixed $\kappa$, the error relative to $f^\lambda/n$ tends to zero
as $n\to\infty$, uniformly over $\lambda\in\Bpart_n$ and all such
weights $w$.
\end{corollary}

\begin{proof}
Since $w(1)=1$, equation~\eqref{eq:general-schur-coefficient}
and \Cref{key:proposition-uniform-balanced} give
\eqref{key:quantitative-balanced-coefficient}.
By \eqref{eq:X-ratio},
$X_{n+1}/X_n\ge(9/5)\sqrt{18/19}>1$ for $n\ge18$.
Thus $X_n$ grows at least geometrically, so both
$X_n^{-1/2}$ and $n^{3/4}X_n^{-2/3}$ tend to zero.
It follows that $\mathcal G_n(\kappa)\to0$ for fixed $\kappa$.
The bound is independent of $\lambda$ and $w$, proving uniformity.
\end{proof}

\begin{corollary}
\label{cor:linear-criterion}
Let $n\ge18$, let $0<\kappa\le2$, and let
$w:\D(n)\to\mathbb R$ satisfy $w(1)=1$.
Assume:
\begin{enumerate}[label=\textup{(\roman*)}]
\item $|\tau_n(w)|,|\eps_n(w)|\le1$;
\item
\[
 |w(2)|\le1\quad(2\mid n),
 \qquad |w(d)|\le\kappa d\quad(d\mid n,\ d\ge3).
\]
\end{enumerate}
Then $[s_\lambda]F_n(w)>0$ for every $\lambda\vdash n$ other than
\[
 (n),\quad(n-1,1),\quad(2,1^{n-2}),\quad(1^n),
\]
whose respective coefficients are
\[
 \tau_n(w),\quad1-\tau_n(w),\quad
 1-\eps_n(w),\quad\eps_n(w).
\]
Consequently, the only possible negative Schur coefficients are those indexed
by $(n)$ and $(1^n)$, and $F_n(w)$ is Schur-positive if and only if
$\tau_n(w),\eps_n(w)\in[0,1]$.
\end{corollary}

\begin{proof}
For $d\ge3$, the inequality $\kappa d\le d^2$ and
\Cref{prop:ramanujan-short-cycle} give
$A_{n,r}(w)<g_r(n)/3<3g_r(n)/5$.
By \eqref{eq:Enw} and \Cref{key:lemma-square-root-divisor-sum},
$E_n(w)\le\mathcal G_n(\kappa)$, and
\Cref{key:proposition-G-numerics} gives $\mathcal G_n(\kappa)<1$.
The hypotheses of \Cref{thm:criterion} therefore hold.
Its four explicit coefficients give the remaining assertions.
\end{proof}

\section{Schur-positive Boolean \texorpdfstring{$Q$}{Q}-sums}\label{sec:boolean}

We prove the Boolean classification and deduce
\cite[Conjectures~1--4]{SundaramPrimePower}, using Sundaram's
plethystic identity for the product statement.
For $1\in J\subseteq\D(n)$, we apply \Cref{thm:criterion} to
$Q_{n,J}=F_n(\psi^J)$.  The bounds below verify its error conditions for
$n\ge18$, while \Cref{lem:boolean-endpoints} gives the remaining coefficients.

\subsection{Bounds for the weights \texorpdfstring{$\psi^J$}{psiJ}}\label{subsec:boolean-bounds}

For $d\ge1$, define
\begin{equation}\label{eq:UVd}
 U_+(d)\coloneqq\sum_{\substack{u\mid\rad(d)\\\mu(u)=1}}\frac1u,
 \qquad
 U_-(d)\coloneqq\sum_{\substack{u\mid\rad(d)\\\mu(u)=-1}}\frac1u.
\end{equation}

\begin{lemma}\label{lem:boolean-weight-bound}
Let $1\in J\subseteq\D(n)$.  For every $d\ge3$,
\begin{equation}\label{eq:boolean-weight-bound}
 |\psi^J(d)|\le dU_+(d)\le d^{4/3}.
\end{equation}
Moreover, $|\psi^J(2)|=1$.
\end{lemma}

\begin{proof}
Writing $e=d/u$ in the definition of $\psi^J(d)$ gives
\[
 \psi^J(d)=d\sum_{u\mid\rad(d)}\frac{\mu(u)}u\one_{d/u\in J}.
\]
The sum of the positive terms is at most $U_+(d)$ and the sum of the absolute
values of the negative terms is at most $U_-(d)$.  Since
\[
 U_+(d)-U_-(d)=\prod_{p\mid d}\left(1-\frac1p\right)>0,
\]
the absolute value of any subsum obtained by selecting terms with
coefficients $0$ or $1$ is at most $U_+(d)$.

If $p$ is prime and $p\nmid d$, then
\[
 U_+(dp)=U_+(d)+\frac{U_-(d)}p
 \le U_+(d)\left(1+\frac1p\right).
\]
Moreover, $U_+(p)=1$ for every prime $p$.
Order the prime divisors of $d$ with $2$ first, when it occurs.
Each subsequent prime satisfies $p\ge3$, so
$1+1/p\le p^{1/3}$.
Since $U_+(d)=U_+(\rad(d))$, iteration gives
\[
 U_+(d)\le\rad(d)^{1/3}\le d^{1/3}.
\]
Finally,
$\psi^J(2)=-1+2\one_{2\in J}\in\set{-1,1}$ because $1\in J$.
\end{proof}

\begin{corollary}\label{prop:boolean-short-cycle}
Let $n\ge18$ and $1\in J\subseteq\D(n)$.  For $2\le r<n/2$,
\begin{equation}\label{eq:boolean-short-cycle}
 A_{n,r}(\psi^J)<\frac13g_r(n).
\end{equation}
\end{corollary}

\begin{proof}
By \Cref{lem:boolean-weight-bound}, $|\psi^J(2)|=1$ and
$|\psi^J(d)|\le d^{4/3}\le d^2$ for $d\ge3$.
Apply \Cref{prop:ramanujan-short-cycle}.
\end{proof}

\begin{lemma}\label{lem:divisor-sum}
For every $n\ge1$,
\begin{equation}\label{eq:divisor-sum}
 \sum_{d\mid n}d^{5/6}\le7n^{11/12}.
\end{equation}
\end{lemma}

\begin{proof}
We apply the divisor-pairing argument of \cite[Lemma~4.5]{Hou} to
the exponent $5/6$.
Let $N\coloneqq\lfloor\sqrt n\rfloor$.  Pair each divisor $d\le\sqrt n$ with
$n/d$, and then enlarge the sum over such $d$ to a sum over all integers
from $1$ to $N$.  This gives
\[
 \begin{aligned}
 \sum_{d\mid n}d^{5/6}
 &\le\sum_{j=1}^N\left(j^{5/6}+\left(\frac nj\right)^{5/6}\right)\\
 &\le N^{11/6}
   +n^{5/6}\left(1+\int_1^N x^{-5/6}\,dx\right)\\
 &\le n^{11/12}+6n^{11/12}
 =7n^{11/12}.
 \end{aligned}
\]
When $n$ is a square, this estimate counts the middle divisor twice, which
only enlarges the upper bound.
\end{proof}

\begin{proposition}\label{prop:boolean-balanced}
For every $n\ge18$ and every $J\subseteq\D(n)$ with $1\in J$,
\begin{equation}\label{eq:boolean-balanced}
 E_n(\psi^J)<1.
\end{equation}
\end{proposition}

\begin{proof}
By \Cref{lem:boolean-weight-bound,lem:divisor-sum},
\begin{equation}\label{eq:boolean-E-bound}
 E_n(\psi^J)\le e^{1/24}\left(
 \frac{X_n^{-1/2}}{\sqrt2}+7n^{11/12}X_n^{-2/3}\right).
\end{equation}
By \Cref{lem:balanced-factor-decrease}, the right-hand side is decreasing
for $n\ge20$.  Its value at $n=20$ is smaller than $1$ by
\Cref{lem:balanced-base-cases}.  The same lemma treats the remaining
cases $n=18$ and $n=19$ directly.
\end{proof}

\begin{remark}
\label{rem:boolean-not-linear}
There is no constant $\kappa$ such that
$|\psi^J(d)|\le\kappa d$ for every $n\ge1$, every
$1\in J\subseteq\D(n)$, and every $d\mid n$.
Let $q_1<q_2<\cdots$ be the odd primes.  For $m\ge1$, let
$d_m\coloneqq q_1\cdots q_{2m}$, take $n=d_m$, and set
\[
 J_m\coloneqq\set{d_m/u:u\mid d_m,\ \mu(u)=1}.
\]
Since $\mu(d_m)=1$, the set $J_m$ contains $1$, and
\[
 \frac{\psi^{J_m}(d_m)}{d_m}
 =\sum_{\substack{u\mid d_m\\\mu(u)=1}}\frac1u
 =\frac12\left(
 \prod_{p\mid d_m}\left(1+\frac1p\right)
 +\prod_{p\mid d_m}\left(1-\frac1p\right)
 \right).
\]
The first product tends to infinity with $m$, by the divergence of the sum
of reciprocals of the primes.  Hence no such constant $\kappa$ exists.
\end{remark}

\subsection{Fixed-degree classification of Boolean \texorpdfstring{$Q$}{Q}-sums}
\label{subsec:boolean-classification}

The coefficient of $p_1^n$ forces $1\in J$ for a nonzero Schur-positive
Boolean sum.  In even degree, \eqref{eq:endpoint-even-Q} also requires
$n\in J\Longrightarrow n/2\in J$.  These conditions are sufficient.

\begin{theorem}
\label{thm:boolean-classification}
Fix $n\ge1$ and let $J\subseteq\D(n)$.  The Boolean $Q$-sum $Q_{n,J}$ is
Schur-positive if and only if either $J=\varnothing$, or
\begin{equation}\label{eq:boolean-condition}
 1\in J
 \quad\text{and, if }2\mid n,\quad
 n\in J\Longrightarrow n/2\in J.
\end{equation}
If $1\in J$, $n$ is even, $n\in J$, and $n/2\notin J$, then
$[s_{(1^n)}]Q_{n,J}=-1$, and every other Schur coefficient is nonnegative.
\end{theorem}

\begin{proof}
The zero sum is Schur-positive.
Suppose that $J\ne\varnothing$ and $Q_{n,J}$ is Schur-positive.
If $1\notin J$, then $Q_{n,J}\ne0$ because the $Q_{n,d}$ form a basis,
but its $p_1^n$-coefficient is zero.  This is impossible:
$[p_1^n]s_\lambda=f^\lambda/n!>0$ for every $\lambda\vdash n$.
Thus $1\in J$.  When $n$ is even, \Cref{lem:boolean-endpoints}
then shows that the sign coefficient is negative if
$n\in J$ but $n/2\notin J$.  This proves necessity.

For sufficiency and the assertion about exceptional supports,
assume $1\in J$.  For $n\ge18$,
\Cref{prop:boolean-short-cycle,prop:boolean-balanced} give conditions
\textup{(ii)} and \textup{(iii)} of \Cref{thm:criterion}.  Put
$a\coloneqq\one_{n\in J}$.  If $n$ is odd, then
\Cref{lem:boolean-endpoints} gives
$\tau_n(\psi^J)=\eps_n(\psi^J)=a$.  If $n$ is even, put
$b\coloneqq\one_{n/2\in J}$.  Then
\[
 \tau_n(\psi^J)=a,
 \qquad
 \eps_n(\psi^J)=b-a.
\]
In both parities, $|\tau_n(\psi^J)|,|\eps_n(\psi^J)|\le1$, so
\Cref{thm:criterion} applies.  The conditions $\tau_n(\psi^J),\eps_n(\psi^J)\in[0,1]$
hold automatically in odd degree and are equivalent to $a\le b$
in even degree.
In the exceptional even case $(a,b)=(1,0)$, the four distinguished coefficients are $1,0,2,-1$, while
every other coefficient is positive.

For $n=1$ the assertion is direct, and for prime $n<18$ it follows from
$\D(n)=\set{1,n}$ and \eqref{eq:zeta-Q-Foulkes}.  For composite $n<18$,
\Cref{prop:boolean-low} proves nonnegativity outside the four
distinguished partitions.  Their coefficients are given by
\Cref{lem:endpoint-coefficients,lem:boolean-endpoints}, so the same
condition is necessary and sufficient.

The preceding arguments also show that, when the even-degree condition
fails and $1\in J$, the sign coefficient is the unique negative coefficient.
\end{proof}

A subset $J\subseteq\D(n)$ is called an \emph{admissible $Q$-support} if
$J=\varnothing$ or if it satisfies \eqref{eq:boolean-condition}.

\begin{corollary}\label{cor:boolean-strictness}
Let $n\ge18$, and let $1\in J\subseteq\D(n)$.
Every Schur coefficient of $Q_{n,J}$ outside
\[
 (n),\quad(n-1,1),\quad(2,1^{n-2}),\quad(1^n)
\]
is positive.  Put $a\coloneqq\one_{n\in J}$.  If $n$ is odd, then
\begin{align*}
 [s_{(n)}]Q_{n,J}&=a,&
 [s_{(n-1,1)}]Q_{n,J}&=1-a,\\
 [s_{(2,1^{n-2})}]Q_{n,J}&=1-a,&
 [s_{(1^n)}]Q_{n,J}&=a.
\end{align*}
If $n$ is even, put $b\coloneqq\one_{n/2\in J}$.  Then
\begin{align*}
 [s_{(n)}]Q_{n,J}&=a,&
 [s_{(n-1,1)}]Q_{n,J}&=1-a,\\
 [s_{(2,1^{n-2})}]Q_{n,J}&=1-b+a,&
 [s_{(1^n)}]Q_{n,J}&=b-a.
\end{align*}
\end{corollary}

\begin{proof}
The positivity outside the four distinguished partitions follows from
\Cref{thm:criterion,prop:boolean-short-cycle,prop:boolean-balanced},
as in the proof of \Cref{thm:boolean-classification}.
The displayed values follow from
\Cref{lem:endpoint-coefficients,lem:boolean-endpoints}.
\end{proof}

\begin{remark}
When $n\ge18$ is even, the respective sets of partitions with zero Schur
coefficient for $(a,b)=(0,0),(0,1),(1,1)$ are
\[
 \set{(n),(1^n)},\qquad
 \set{(n),(2,1^{n-2})},\qquad
 \set{(n-1,1),(1^n)}.
\]
\end{remark}

\subsection{All-degree classifications and Sundaram's conjectures}
\label{subsec:special-families}
\begin{corollary}\label{cor:global-boolean}
Let $T\subseteq\mathbb Z_{\ge1}$, and let $f_n^T$ be defined by
\eqref{eq:psi-fT}.  Then $f_n^T$ is Schur-positive for every $n\ge1$
if and only if $T=\varnothing$, or
\begin{equation}\label{eq:global-halving}
 1\in T
 \qquad\text{and}\qquad
 2m\in T\Longrightarrow m\in T\quad(m\ge1).
\end{equation}
\end{corollary}

\begin{proof}
Put $J_n\coloneqq T\cap\D(n)$, so $f_n^T=Q_{n,J_n}$ by
\eqref{eq:fT-Q-sum}.  If $T=\varnothing$, all these functions vanish.
Suppose that $T\ne\varnothing$ and every $f_n^T$ is Schur-positive.
Choose $k\in T$.  Then $J_k\ne\varnothing$, and
\Cref{thm:boolean-classification} gives $1\in J_k\subseteq T$.
If $2m\in T$, the same theorem applied in degree $2m$ gives $m\in T$.
Conversely, if \eqref{eq:global-halving} holds, every $J_n$ contains $1$,
and $n/2\in J_n$ whenever $n$ is even and $n\in J_n$.
The same classification therefore proves Schur positivity in every degree.
\end{proof}

Let
\[
 \widehat\Lambda_{\mathbb Q}\coloneqq
 \prod_{m\ge0}\Lambda_{\mathbb Q}^m,
 \qquad
 \mathsf H\coloneqq\sum_{m\ge0}h_m\in\widehat\Lambda_{\mathbb Q}.
\]
We use plethystic substitution in this degree completion.
Here $t$ is a formal variable, with the plethystic convention
$p_r[f t^m]=p_r[f]t^{rm}$ for every symmetric function $f$,
$r\ge1$, and $m\ge0$.

\begin{corollary}\label{cor:product-implication}
If $T=\varnothing$, or $1\in T$ and $2m\in T$ implies $m\in T$, then
$\prod_{d\in T}(1-p_dt^d)^{-1}$ is Schur-positive in every degree.
\end{corollary}
\begin{proof}
Combine Corollary~\ref{cor:global-boolean} with Sundaram's identity
$\mathsf H[\sum_{n\ge1}f_n^Tt^n]=\prod_{d\in T}(1-p_dt^d)^{-1}$
\cite[Theorem~6.4 and Corollary~6.5]{SundaramVariations}.
\end{proof}

\begin{remark}\label{rem:Hou-interval}
Let $K\ge1$ and $T=\{1,2,\ldots,K\}$. Then $1\in T$, and $2m\in T$ implies
$m\in T$. Corollary~\ref{cor:global-boolean} therefore gives
\[
 f_n^T=\sum_{\substack{d\mid n\\d\le K}}Q_{n,d}\ge_s0
 \qquad(n,K\ge1).
\]
This recovers Hou's bounded-interval theorem
\cite[Theorem~1.1]{Hou}, which proves
\cite[Conjecture~4]{SundaramPrimePower}.
By \Cref{cor:product-implication}, every coefficient of
$\prod_{d=1}^K(1-p_dt^d)^{-1}$ is also Schur-positive.
More generally, every nonempty divisor-closed set satisfies
\eqref{eq:global-halving}.
Closure under all divisors is not necessary: $T=\{1,9\}$ contains $1$
and is closed under halving, so every $f_n^T$ is Schur-positive.
\end{remark}

\begin{corollary}
\label{cor:one-k-classification}
Let $k\ge2$. The function $f_n^{\{1,k\}}$ fails to be Schur-positive
exactly when $n=k$ and $k>2$ is even. In that degree the coefficient
of $s_{(1^n)}$ is $-1$ and every other Schur coefficient is nonnegative.
\end{corollary}

\begin{proof}
The set $J=\{1,k\}\cap\D(n)$ contains $1$.  The even-degree condition in
\Cref{thm:boolean-classification} fails precisely when $n=k>2$ is even:
then $n\in J$ but $n/2\notin J$.  The same theorem gives the asserted
Schur coefficients in the exceptional degree.
\end{proof}

For $k\ge2$, set $T_k\coloneqq\{k^j:j\ge0\}$.

\begin{corollary}
\label{cor:powers-classification}
For $k\ge2$, the function $f_n^{T_k}$ fails
to be Schur-positive exactly when $k>2$ is even and $n=k^a$ for
some $a\ge1$. In every exceptional degree the coefficient of $s_{(1^n)}$ is
$-1$ and all remaining coefficients are nonnegative.
\end{corollary}

\begin{proof}
Apply \Cref{thm:boolean-classification} to $J=T_k\cap\D(n)$.
An exceptional degree must be even and satisfy $n=k^a$ for some $a\ge1$.
For such $n$, the relation $n/2\in T_k$ holds if and only if $k=2$.
This gives exactly the stated exceptions and their Schur coefficients.
\end{proof}

\begin{corollary}\label{key:corollary-powers-product}
For $k=2$ and every odd $k\ge3$,
\[
 \prod_{j\ge0}(1-p_{k^j}t^{k^j})^{-1}\ge_s0.
\]
Its specialization at $t=1$ is well defined in the degree completion
and is Schur-positive there.
\end{corollary}

\begin{proof}
For the stated values of $k$, the set $T_k$ contains $1$ and is closed
under halving its even elements.  Thus \Cref{cor:product-implication}
applies.  In each degree, only finitely many factors contribute, so the
specialization at $t=1$ is well defined in $\widehat\Lambda_{\mathbb Q}$.
\end{proof}

\Cref{cor:powers-classification,cor:one-k-classification} prove
\cite[Conjectures~1 and~3]{SundaramPrimePower}, and
\Cref{key:corollary-powers-product} proves Conjecture~2.
Conjecture~4, previously proved by Hou, is recovered in
\Cref{rem:Hou-interval}.

For prime $k=q$, the classifications also follow from
$f_n^{\{1,q\}}=\ell_n^{(q)}$ when $q\mid n$ and
$f_n^{T_q}=\ell_n^{(q^a)}$, where $a\in\mathbb Z_{\ge0}$ is maximal
with $q^a\mid n$.
For composite $k\mid n$, neither
$\{1,k\}$ nor $T_k\cap\D(n)$ is a principal order ideal in $\D(n)$, so
each corresponding Foulkes expansion has a negative coordinate even in
degrees in which the function is Schur-positive.

We now compare Schur positivity with the stronger requirement that
$f_n^T$ have nonnegative Foulkes coordinates in every degree.
Recall the cone $\mathcal F_n^+$ from \Cref{subsec:global-lifts}.
For $T\subseteq\mathbb N$ with $1\in T$ and $n\ge1$, put
\begin{equation}\label{eq:rT-definition}
 r_T(n)\coloneqq\operatorname{lcm}(T\cap\D(n)).
\end{equation}

\begin{corollary}\label{key:corollary-global-foulkes-positive}
For $T\subseteq\mathbb N$, the following are equivalent:
\begin{enumerate}[label=\textup{(\roman*)}]
\item $f_n^T\in\mathcal F_n^+$ for every $n\ge1$;
\item either $T=\varnothing$ or $T$ is closed under taking divisors and
least common multiples;
\item either $T=\varnothing$ or there are numbers
$c_p\in\mathbb Z_{\ge0}\cup\{\infty\}$, indexed by the primes $p$, such that
\[
 T=\{m\ge1:\nu_p(m)\le c_p\text{ for every prime }p\},
\]
where $\nu_p(m)$ is the exponent of $p$ in $m$.
\end{enumerate}
If $T$ is nonempty and satisfies these equivalent conditions, then
\[
 r_T(n)=\prod_p p^{\min\{\nu_p(n),c_p\}}
\]
and
\[
 f_n^T=\ell_n^{(r_T(n))}\qquad(n\ge1),
\]
with the convention $\min\{a,\infty\}=a$.
\end{corollary}

\begin{proof}
The assertions are immediate for $T=\varnothing$, so assume $T\ne\varnothing$.
Under \textup{(i)}, applying \Cref{prop:fixed-foulkes} in degree $m\in T$
gives $T\cap\D(m)=\D(m)$; hence $T$ is closed under divisors.
For $a,b\in T$, apply the same proposition in degree
$m=\operatorname{lcm}(a,b)$.  Since $T\cap\D(m)=\D(r)$ contains both
$a$ and $b$, one has $m\mid r\mid m$, and therefore $m\in T$.
This proves \textup{(ii)}.

Under \textup{(ii)}, take $c_p=\sup\{\nu_p(m):m\in T\}$.
By closure under divisors, $T$ contains every permitted prime power;
by closure under least common multiples, it contains every finite
product of such prime powers.  This proves \textup{(iii)}, whose converse is direct.
Finally, \textup{(ii)} gives
$T\cap\D(n)=\D(r_T(n))$ for every $n$, proving \textup{(i)} by
\Cref{prop:fixed-foulkes}.  The description in \textup{(iii)} gives the
stated formula for $r_T(n)$, and \eqref{eq:zeta-Q-Foulkes} gives that for
$f_n^T$.
\end{proof}

Thus, for a nonempty set $T$, nonnegative Foulkes coordinates in every degree
require closure under all divisors and least common multiples, whereas
Schur positivity requires only $1\in T$ and closure under halving.

\begin{remark}\label{key:remark-global-foulkes-examples}
The following sets satisfy the equivalent conditions in
\cref{key:corollary-global-foulkes-positive}.
If $T=\langle S\rangle$ is the multiplicative monoid generated by
a set $S$ of primes, then $c_p=\infty$ for $p\in S$ and $c_p=0$ otherwise.
If $T=\D(K)$ for a fixed positive integer $K$,
then
\[
 r_T(n)=\gcd(K,n),\qquad f_n^T=\ell_n^{(\gcd(K,n))}.
\]
If $T$ is the set of squarefree positive integers, then
\[
 r_T(n)=\rad(n),\qquad
 f_n^T=\ell_n^{(\rad(n))}.
\]
More generally, fix an integer $h\ge2$, and let $T$ be the set of positive
integers not divisible by $p^h$ for any prime $p$.  Then
\[
 r_T(n)=\prod_p p^{\min\{\nu_p(n),h-1\}}.
\]
\end{remark}

\section{The Boolean region and its integral decompositions}\label{sec:boolean-geometry}\label{subsec:boolean-polytope}

By \Cref{thm:boolean-classification}, every nonzero Schur-positive Boolean
sum has first $Q$-coordinate $1$.  Let $H_n$ be the convex hull of their
$Q$-coordinate vectors:
\begin{equation}\label{eq:Hn-definition}
 H_n\coloneqq\conv\set{\one_J:\varnothing\ne J\subseteq\D(n),\
 Q_{n,J}\text{ is Schur-positive}}.
\end{equation}
Its conical hull is
\begin{equation}\label{eq:cone-Hn}
 \operatorname{cone}(H_n)
 =\set{t x:t\in\mathbb R_{\ge0},\ x\in H_n}
 \subseteq\mathbb R^{\D(n)}.
\end{equation}
This cone is pointed.  By \Cref{thm:Q-lattice-saturation}, the map
$\mathcal Q_n$ identifies $\mathbb Z^{\D(n)}$ with the integral character
lattice $\Rspace_{n,\mathbb R}\cap\Lambda_{\mathbb Z}^n$.
The \emph{Hilbert basis} of a pointed rational
polyhedral cone $C\subseteq\mathbb R^{\D(n)}$, with respect to
$\mathbb Z^{\D(n)}$, is the unique minimal subset
\[
 B\subseteq\bigl(C\cap\mathbb Z^{\D(n)}\bigr)\setminus\set{0}
\]
whose nonnegative integral span is $C\cap\mathbb Z^{\D(n)}$.  Equivalently,
its elements are the nonzero lattice points of $C$ that cannot be written
as sums of two nonzero lattice points of $C$.

\begin{theorem}\label{thm:boolean-core}
For every $n\ge1$,
\begin{equation}\label{eq:boolean-cube-intersection}
 P_n^Q\cap\set{x:0\le x_d\le1\text{ for all }d>1}=H_n.
\end{equation}
In particular,
\[
 \Delta_n\subseteq H_n\subseteq P_n^Q.
\]
For odd $n$, the polytope $H_n$ is a cube, and $H_2$ is an interval.
For even $n\ge4$, it is the product of a cube with the triangle
\[
 \set{(a,b):0\le a\le b\le1}
\]
in the coordinates $(x_n,x_{n/2})$.
\end{theorem}

\begin{proof}
The vertices of the cube
$\set{x\in\mathbb R^{\D(n)}:x_1=1,\ 0\le x_d\le1\text{ for }d>1}$
are the vectors $\one_J$ with $1\in J\subseteq\D(n)$.  If $n$ is odd,
every such vertex belongs to
$P_n^Q$ by \Cref{thm:boolean-classification}; convexity therefore gives the
whole cube.

For $n=2$, we have $Q_{2,1}=s_{(1^2)}$ and
$Q_{2,2}=s_{(2)}-s_{(1^2)}$.  Hence
\[
 P_2^Q=H_2=\set{(1,x_2):0\le x_2\le1}.
\]

Now suppose that $n\ge4$ is even.  Formula \eqref{eq:endpoint-even-Q}
shows that every point of $P_n^Q$ satisfies $x_n\le x_{n/2}$.  The vertices
of the coordinate cube cut out by this inequality are precisely the Boolean
vectors satisfying $n\in J\Rightarrow n/2\in J$.  They all belong to
$P_n^Q$ by \Cref{thm:boolean-classification}, so their convex hull is the
full intersection in \eqref{eq:boolean-cube-intersection}.  Finally, the
indicator vector of every principal order ideal satisfies the same
condition, which gives $\Delta_n\subseteq H_n$.
\end{proof}

For a lattice polytope $P\subseteq\mathbb R^N$, its \emph{Ehrhart polynomial}
is the counting function
\[
 L_P(m)\coloneqq\#\bigl(mP\cap\mathbb Z^N\bigr)
 \qquad(m\in\mathbb Z_{\ge0})
\]
which is a polynomial in $m$ by Ehrhart's theorem
\cite[Chapter~3]{BeckRobins}.  The polytope $P$ has the
\emph{integer decomposition property} if every point of
$mP\cap\mathbb Z^N$, $m\in\mathbb Z_{\ge1}$, is a sum of $m$ points of
$P\cap\mathbb Z^N$.

\begin{corollary}\label{cor:boolean-Ehrhart}
For every $n\ge1$, the polytope $H_n$ has the integer decomposition
property.  Its Ehrhart
polynomial is
\[
 L_{H_n}(m)=(m+1)^{|\D(n)|-1}\qquad(n\text{ odd}),
\]
and
\[
 L_{H_n}(m)=(m+1)^{|\D(n)|-2}\frac{m+2}{2}
 \qquad(n\ge4\text{ even}).
\]
For $n=2$, $L_{H_2}(m)=m+1$.  Moreover,
\[
 \vol(H_n)=
 \begin{cases}
  1,&n\text{ odd or }n=2,\\
  1/2,&n\ge4\text{ even}
 \end{cases}.
\]
The Hilbert basis of $\operatorname{cone}(H_n)$ with respect to
the standard $Q$-coordinate lattice $\mathbb Z^{\D(n)}$ is
\[
 \set{\one_J:\varnothing\ne J\subseteq\D(n),\
 Q_{n,J}\text{ is Schur-positive}}.
\]
\end{corollary}

\begin{proof}
In the cube cases, each free coordinate of a lattice point of $mH_n$
has $m+1$ choices.  For even $n\ge4$, the ordered pair
$0\le x_n\le x_{n/2}\le m$ has
$\sum_{b=0}^m(b+1)=(m+1)(m+2)/2$ choices, and each of the remaining
$|\D(n)|-3$ free coordinates has $m+1$ choices.  This proves the Ehrhart
formulas.  The volume formulas follow from the unit cube and the
triangle of area $1/2$.
For $x\in mH_n\cap\mathbb Z^{\D(n)}$, one
has $x_1=m$ and $0\le x_d\le m$.  For $1\le j\le m$, set
\[
 J_j\coloneqq\set{d\mid n:x_d\ge j}.
\]
Then $x=\sum_{j=1}^m\one_{J_j}$.  Each $J_j$ contains $1$, and in even
degree $x_n\le x_{n/2}$ implies $n\in J_j\Rightarrow n/2\in J_j$.
Thus every $\one_{J_j}$ belongs to $H_n$, proving the
integer decomposition property.

Now let $y\in\operatorname{cone}(H_n)\cap\mathbb Z^{\D(n)}$ be nonzero and
put $m\coloneqq y_1$.  Since every point of $H_n$ has first coordinate $1$,
one has $m\in\mathbb Z_{\ge1}$ and $y\in mH_n$.
The same decomposition shows that the displayed Boolean vectors generate
all lattice points of the cone.  None is a sum of two nonzero lattice points
of the cone: each has first coordinate $1$, whereas every nonzero lattice
point of the cone has a positive integer first coordinate.  This proves the
Hilbert-basis assertion.
\end{proof}

Combining \Cref{prop:Delta-unimodular,cor:boolean-Ehrhart} and writing
$r_n\coloneqq|\D(n)|-1$, we obtain
\begin{equation}\label{eq:volume-ratio}
 \frac{\vol(H_n)}{\vol(\Delta_n)}=
 \begin{cases}
  r_n!,&n\text{ odd or }n=2,\\
  r_n!/2,&n\ge4\text{ even}.
 \end{cases}
\end{equation}

Since the lattice points of $H_n$ are its Boolean vertices, setting
$m=1$ in \Cref{cor:boolean-Ehrhart} shows that the number of
Schur-positive Boolean $Q$-sums, including the zero sum, is
\[
 1+L_{H_n}(1)=
 \begin{cases}
  1+2^{|\D(n)|-1},&n\text{ odd or }n=2,\\
  1+3\cdot2^{|\D(n)|-3},&n\ge4\text{ even}.
 \end{cases}
\]

Together with \Cref{cor:boolean-strictness}, the integer decomposition
property yields the following bounds for major-index residue sums.

\begin{corollary}
\label{cor:weighted-boolean-residue}
Let $h\in\mathbb Z_{\ge1}$ and let
$x\in hH_n\cap\mathbb Z^{\D(n)}$.  Then, for every
$\lambda\vdash n$,
\begin{equation}\label{eq:weighted-boolean-nonnegative}
 \sum_{e\mid n}\beta_x(e)a_{\lambda,e}
 =[s_\lambda]\mathcal Q_n(x)\ge0.
\end{equation}
Equivalently, \eqref{eq:weighted-residue-inequality} holds with
$\beta_x^+$ and $\beta_x^-$ as defined in
\eqref{eq:weighted-beta-parts}.  If $n\ge18$, then the left-hand side of
\eqref{eq:weighted-boolean-nonnegative} is at least $h$ for every
$\lambda\vdash n$ outside the four distinguished partitions.  For odd
$n\ge4$, the distinguished coefficients are
\begin{align*}
 [s_{(n)}]\mathcal Q_n(x)&=x_n,&
 [s_{(n-1,1)}]\mathcal Q_n(x)&=h-x_n,\\
 [s_{(2,1^{n-2})}]\mathcal Q_n(x)&=h-x_n,&
 [s_{(1^n)}]\mathcal Q_n(x)&=x_n.
\end{align*}
For even $n\ge4$, they are
\begin{align*}
 [s_{(n)}]\mathcal Q_n(x)&=x_n,&
 [s_{(n-1,1)}]\mathcal Q_n(x)&=h-x_n,\\
 [s_{(2,1^{n-2})}]\mathcal Q_n(x)&=h-x_{n/2}+x_n,&
 [s_{(1^n)}]\mathcal Q_n(x)&=x_{n/2}-x_n.
\end{align*}
\end{corollary}

\begin{proof}
Use the decomposition
$x=\sum_{j=1}^h\one_{J_j}$ from the proof of
\Cref{cor:boolean-Ehrhart}; every $Q_{n,J_j}$ is Schur-positive.
Equation~\eqref{eq:weighted-major-index-coefficient} therefore proves
\eqref{eq:weighted-boolean-nonnegative}.  When $n\ge18$, each summand has
a positive integer coefficient at every other partition by
\Cref{cor:boolean-strictness}, so their sum is at least $h$.
Since $x_1=h$, the four stated values follow by applying
\Cref{lem:endpoint-coefficients} and \eqref{eq:endpoint-Q}
to $\mathcal Q_n(x)/h$ and multiplying by $h$.
\end{proof}

\begin{remark}
\label{rem:weighted-one-k}
Let $k>1$ divide $n$, let $h\in\mathbb Z_{\ge1}$, and let
$0\le m\le h$ be integral.  Assume that the subset
$\set{1,k}\subseteq\D(n)$ is an admissible $Q$-support; equivalently, either
$k<n$, or $n=k$ is odd, or $n=k=2$.  Then
\[
 x\coloneqq h\one_{\set{1}}+m\one_{\set{k}}
 =(h-m)\one_{\set{1}}+m\one_{\set{1,k}}
 \in hH_n\cap\mathbb Z^{\D(n)},
\]
and \Cref{cor:weighted-boolean-residue} gives
\begin{equation}\label{eq:weighted-one-k-residue}
 h a_{\lambda,1}
 +m\sum_{e\mid k}\mu(k/e)a_{\lambda,e}\ge0
 \qquad(\lambda\vdash n).
\end{equation}
The case $h=m=1$ is the inequality in
\Cref{key:corollary-major-index-inequalities}.
For this case, when $k$ is composite and $n>k$,
\Cref{cor:composite-equality} determines the equality cases and,
for $n\ge18$, gives quantitative bounds for balanced partitions.
\end{remark}

\section{The Ramanujan-square family}\label{sec:ramanujan}

The character estimates used for the Boolean classification also
resolve the Ramanujan-square conjecture of Shareshian and Sundaram.
Define
\begin{equation}\label{eq:Ramanujan-square}
 R_n^{\sq}\coloneqq\sum_{d\mid n}c_d(n/d)^2p_d^{n/d},
 \qquad
 w_n:\D(n)\longrightarrow\mathbb Z_{\ge0},
 \quad w_n(d)\coloneqq c_d(n/d)^2.
\end{equation}
Then
\begin{equation}\label{eq:rho-Fnw}
 \widetilde R_n^{\sq}\coloneqq\frac1nR_n^{\sq}=F_n(w_n).
\end{equation}
Shareshian and Sundaram conjectured that $R_n^{\sq}$ is Schur-positive
for every $n$ \cite[Conjecture~38]{ShareshianSundaram}.  They proved this
when $n$ is squarefree or four times an odd squarefree integer
\cite[Theorem~36]{ShareshianSundaram} and reported further evidence
in \cite[Table~1]{ShareshianSundaram}.
\Cref{thm:ramanujan-square} proves the conjecture and determines all zero
coefficients.  The examples in degrees $8$ and $9$ show that the resulting
normalized coordinate vectors need not belong to $H_n$.

\subsection{Arithmetic properties of the divisor weights}\label{subsec:ramanujan-arithmetic}

Put
\begin{equation}\label{eq:tn}
 t(n)\coloneqq\sum_{d\mid n}w_n(d)=\sum_{d\mid n}c_d(n/d)^2.
\end{equation}
For $n\ge2$, the formulas in
\cite[equations~(24)--(25)]{ShareshianSundaram} identify
$t(n)$ and $n-t(n)$ as the trivial and standard Schur coefficients,
respectively, of $R_n^{\sq}$.
We write $p^A\parallel n$ when $p^A\mid n$ and $p^{A+1}\nmid n$.

For a prime $p$ and an integer $A\ge1$, define
\[
 T_p(A)\coloneqq\sum_{a=0}^A c_{p^a}(p^{A-a})^2.
\]

\begin{lemma}\label{lem:local-tn}
For every prime $p$ and every integer $A\ge1$, one has
\begin{equation}\label{eq:TpA}
 T_p(A)=
 \begin{cases}
  \displaystyle\frac{(p-1)p^A+2}{p+1},&A\text{ even},\\[2mm]
  \displaystyle\frac{2(p^A+1)}{p+1},&A\text{ odd}.
 \end{cases}
\end{equation}
Moreover, for every $n\ge1$,
\begin{equation}\label{eq:tn-product}
 t(n)=\prod_{p^A\parallel n}T_p(A),
\end{equation}
and
\begin{equation}\label{eq:tn-bound}
 0<t(n)\le n.
\end{equation}
The upper bound in \eqref{eq:tn-bound} is attained if and only if
$n\in\{1,2\}$.
\end{lemma}

\begin{proof}
By \eqref{eq:ramanujan-standard},
\[
 c_{p^a}(p^{A-a})=
 \begin{cases}
  \phi(p^a),&2a\le A,\\
  -p^{a-1},&2a=A+1,\\
  0,&2a>A+1.
\end{cases}
\]
If $A=2m$, then
\[
 T_p(A)=1+(p-1)^2\sum_{j=0}^{m-1}p^{2j}
 =\frac{(p-1)p^A+2}{p+1}.
\]
If $A=2m+1$, then
\[
 T_p(A)=1+(p-1)^2\sum_{j=0}^{m-1}p^{2j}+p^{2m}
 =\frac{2(p^A+1)}{p+1}.
\]
Equation~\eqref{eq:tn-product} is the specialization of
\cite[Lemma~32]{ShareshianSundaram} to the Ramanujan-square weights.

For even $A$,
\[
 p^A-T_p(A)=\frac{2(p^A-1)}{p+1}\ge0,
\]
while for odd $A$,
\[
 p^A-T_p(A)=\frac{(p-1)p^A-2}{p+1}\ge0.
\]
Equality in the latter inequality holds only for $(p,A)=(2,1)$.  Multiplication of the
local inequalities proves the assertion.
\end{proof}

The general formulas for the four distinguished coefficients appear in
\cite[equations~(24)--(25) and the following paragraph]{ShareshianSundaram}.
The next lemma evaluates $\tau_n(w_n)$ and $\eps_n(w_n)$.

\begin{lemma}\label{lem:ramanujan-endpoints}
For the weight $w_n$ in \eqref{eq:Ramanujan-square},
\begin{equation}\label{eq:ramanujan-endpoints}
 \tau_n(w_n)=\frac{t(n)}n,
 \qquad
 \eps_n(w_n)=
 \begin{cases}
  t(n)/n,&n\not\equiv2\pmod4,\\
  0,&n\equiv2\pmod4.
 \end{cases}
\end{equation}
\end{lemma}

\begin{proof}
The first formula is immediate.  If $n$ is odd, every sign in
\eqref{eq:tau-epsilon} is positive.  Suppose that $4\mid n$.
Write $n=2^Am$ with $m$ odd and $A\ge2$.
A negative sign can occur only when $n/d$ is odd, which implies
$2^A\mid d$.  By multiplicativity, the factor
$c_{2^A}(n/d)=\mu(2^A)$ occurs in $c_d(n/d)$ and is zero.  Hence every
negatively signed term has weight zero, and the signed and unsigned sums
agree.

Finally, let $n=2m$ with $m$ odd.  Pair the divisors $e$ and $2e$ for
$e\mid m$.  Since $e$ and $m/e$ are odd,
\[
 c_{2e}(m/e)=c_2(m/e)c_e(m/e)=-c_e(m/e),
 \qquad
 c_e(2m/e)=c_e(m/e).
\]
The two squared values are equal, whereas their signs in
\eqref{eq:tau-epsilon} are opposite.  Thus the signed sum is zero.
\end{proof}

\begin{lemma}\label{lem:ramanujan-weight-bounds}
For every $d\mid n$,
\begin{equation}\label{eq:ramanujan-pointwise}
 0\le w_n(d)\le d^2.
\end{equation}
Moreover,
\[
 \sum_{d\mid n}w_n(d)=t(n)\le n.
\]
If $2\mid n$, then $w_n(2)=1$.
\end{lemma}

\begin{proof}
The estimate $|c_d(r)|\le\phi(d)$ is proved in
\cite[Proposition~6]{ShareshianSundaram}; together with
$\phi(d)\le d$ it gives the pointwise bound.  The identity for the sum and
its upper bound follow from \Cref{lem:local-tn}.  If $2\mid n$, then
$c_2(n/2)=\pm1$, so $w_n(2)=1$.
\end{proof}

\begin{remark}
\label{rem:ramanujan-not-linear}
No estimate of the form $w_n(d)\le\kappa d$, with $\kappa$ independent of
$n$ and $d$, holds uniformly for the Ramanujan-square weights.
For $n=p^2$ and $d=p$,
\[
 w_{p^2}(p)=c_p(p)^2=(p-1)^2,
\]
so $w_{p^2}(p)/p$ is unbounded as $p$ varies.
\end{remark}

\subsection{Balanced-partition estimate}
\label{subsec:ramanujan-estimates}

To verify the balanced-partition hypothesis in \Cref{thm:criterion},
we use the total weight bound $t(n)\le n$.

For $n\ge18$, define
\begin{equation}\label{eq:Esq}
 \mathcal E_n^{\sq}\coloneqq e^{1/24}\left(
 \frac{X_n^{-1/2}}{\sqrt2}+nX_n^{-2/3}\right).
\end{equation}

\begin{proposition}\label{prop:ramanujan-balanced}
For every $n\ge18$,
\[
 E_n(w_n)\le\mathcal E_n^{\sq}<1.
\]
\end{proposition}

\begin{proof}
By \Cref{lem:ramanujan-weight-bounds},
\[
 \sum_{\substack{d\mid n\\3\le d<n}}
 \frac{w_n(d)}{\sqrt d}
 \le\sum_{d\mid n}w_n(d)\le n,
\]
which gives $E_n(w_n)\le\mathcal E_n^{\sq}$.  By
\Cref{lem:balanced-factor-decrease}, the quantity
$\mathcal E_n^{\sq}$ is decreasing for $n\ge18$.  Finally,
\Cref{lem:balanced-base-cases} gives $\mathcal E_{18}^{\sq}<1$.
\end{proof}

\subsection{Schur positivity}\label{subsec:ramanujan-positivity}

\begin{theorem}\label{thm:ramanujan-square}
For every $n\ge1$ and every $\lambda\vdash n$ with
$\lambda\ne(1^n)$, one has
\[
 [s_\lambda]R_n^{\sq}>0.
\]
The coefficient of $s_{(1^n)}$ satisfies
\[
 [s_{(1^n)}]R_n^{\sq}=0
 \quad\Longleftrightarrow\quad n\equiv2\pmod4,
\]
and is positive otherwise.  In particular, $R_n^{\sq}$ is Schur-positive
for every $n\ge1$.
\end{theorem}

\begin{proof}
Suppose first that $n\ge18$.
By \Cref{lem:ramanujan-weight-bounds,prop:ramanujan-short-cycle},
$A_{n,r}(w_n)<g_r(n)/3$ for $2\le r<n/2$, and
\Cref{prop:ramanujan-balanced} gives $E_n(w_n)<1$.
Together with $w_n(1)=1$ and
$0\le\tau_n(w_n),\eps_n(w_n)\le1$ from
\Cref{lem:local-tn,lem:ramanujan-endpoints},
these verify all the hypotheses of \Cref{thm:criterion}.
Every coefficient outside the four distinguished partitions is positive.  At the distinguished partitions,
\Cref{lem:endpoint-coefficients,lem:ramanujan-endpoints} give, after
multiplication by $n$,
\[
 t(n),\quad n-t(n),\quad n-n\eps_n(w_n),\quad n\eps_n(w_n).
\]
By \Cref{lem:local-tn}, the first two are positive.  If
$n\equiv2\pmod4$, the last coefficient is zero and the third is $n$;
otherwise the last two coefficients are $n-t(n)$ and $t(n)$, respectively.

For $n<18$, \Cref{prop:ramanujan-low} proves the assertion using
the known nonnegative Foulkes expansions, the estimates of
\Cref{sec:criterion}, and explicit character formulas in the
remaining degrees $8$ and $9$.
\end{proof}

\subsection{Location in the normalized slice}
\label{subsec:ramanujan-consequences}

We now locate $\widetilde R_n^{\sq}$ in the normalized slice $P_n$
of \eqref{eq:Pn-definition} and compute the integral $Q$-coordinates
of $R_n^{\sq}$.

When $n\equiv2\pmod4$, define
\[
 P_n^{\mathrm{sgn}}
 \coloneqq P_n\cap\set{F:[s_{(1^n)}]F=0}.
\]

\begin{corollary}\label{cor:rho-location}
For every $n\ge1$, if $n\not\equiv2\pmod4$, then
$\widetilde R_n^{\sq}\in\operatorname{relint}(P_n)$.  If $n\equiv2\pmod4$, then
$P_n^{\mathrm{sgn}}$ is an exposed face of $P_n$, and
$\widetilde R_n^{\sq}\in\operatorname{relint}(P_n^{\mathrm{sgn}})$.
\end{corollary}

\begin{proof}
At $\widetilde R_n^{\sq}$, \Cref{thm:ramanujan-square} gives strictness of
all Schur inequalities unless $n\equiv2\pmod4$, when only the sign
coefficient vanishes.  In the latter case, the sign coefficient is
nonnegative on $P_n$ and exposes $P_n^{\mathrm{sgn}}$.  Strictness of the
remaining inequalities gives the relative-interior assertions.
\end{proof}

The $Q$-coordinate of $\widetilde R_n^{\sq}$ at a divisor $d$ is, by
\eqref{eq:weight-Q-transform},
\begin{equation}\label{eq:rho-Q-coordinate}
 x_d(\widetilde R_n^{\sq})=\frac1d\sum_{e\mid d}c_e(n/e)^2.
\end{equation}

For $d\mid n$, define
\[
 \xi_{n,d}^{\sq}\coloneqq\frac nd\sum_{e\mid d}c_e(n/e)^2,
 \qquad
 \xi_n^{\sq}\coloneqq(\xi_{n,d}^{\sq})_{d\mid n}.
\]

\begin{corollary}\label{cor:ramanujan-Q-lattice}
For every $n\ge1$, the $Q$-coordinate expansion of $R_n^{\sq}$ is
\[
 R_n^{\sq}=\sum_{d\mid n}\xi_{n,d}^{\sq}Q_{n,d}.
\]
Moreover,
\[
 \xi_n^{\sq}\in nP_n^Q\cap\mathbb Z_{\ge0}^{\D(n)},
 \qquad \xi_{n,1}^{\sq}=n.
\]
\end{corollary}

\begin{proof}
The coordinate formula is $n$ times \eqref{eq:rho-Q-coordinate}.  Each
$\xi_{n,d}^{\sq}$ is integral because $d\mid n$ and every Ramanujan sum is
an integer.  By \Cref{thm:ramanujan-square}, the $Q$-coordinate vector of
$\widetilde R_n^{\sq}$ belongs to $P_n^Q$, so $\xi_n^{\sq}\in nP_n^Q$.
\end{proof}

For $e\mid n$, define
\begin{equation}\label{eq:ramanujan-B-coordinate}
 B_n(e)\coloneqq
 \sum_{\substack{d\mid n\\e\mid d}}\mu(d/e)\xi_{n,d}^{\sq}.
\end{equation}

\begin{corollary}
\label{cor:ramanujan-weighted-residue}
For every $n\ge1$ and every $e\mid n$, one has $B_n(e)\in\mathbb Z$, and
\begin{equation}\label{eq:ramanujan-Foulkes-expansion}
 R_n^{\sq}=\sum_{e\mid n}B_n(e)\ell_n^{(e)}.
\end{equation}
Moreover, for every $\lambda\vdash n$,
\begin{equation}\label{eq:ramanujan-weighted-inequality}
 \sum_{e\mid n}B_n(e)a_{\lambda,e}\ge0.
\end{equation}
For a given partition $\lambda\vdash n$, equality in
\eqref{eq:ramanujan-weighted-inequality} holds if and only if
$n\equiv2\pmod4$ and $\lambda=(1^n)$.
\end{corollary}

\begin{proof}
The vector $\xi_n^{\sq}$ is integral by
\Cref{cor:ramanujan-Q-lattice}.  Applying
\Cref{prop:weighted-major-index-dictionary} with $x=\xi_n^{\sq}$ shows that
the numbers in \eqref{eq:ramanujan-B-coordinate} are integers and proves
\eqref{eq:ramanujan-Foulkes-expansion}.  It also identifies the left-hand
side of \eqref{eq:ramanujan-weighted-inequality} with
$[s_\lambda]R_n^{\sq}$.  The positivity and equality statements now follow
from \Cref{thm:ramanujan-square}.
\end{proof}

Equations~\eqref{eq:canonical-lift} and
\eqref{eq:ramanujan-Foulkes-expansion} give
\[
 \sigma_n(R_n^{\sq})=\sum_{e\mid n}B_n(e)[V_e].
\]
These lift coordinates need not be nonnegative.

\begin{example}\label{ex:ramanujan-eight}
In degree $8$, \eqref{eq:rho-Q-coordinate} gives
$x_4(\widetilde R_8^{\sq})=(1+1+4)/4=3/2$, so the $Q$-coordinate
vector of $\widetilde R_8^{\sq}$ lies outside $H_8$.
Substitution in \eqref{eq:ramanujan-B-coordinate} gives the expansion
\[
 R_8^{\sq}=-4\ell_8^{(2)}+6\ell_8^{(4)}+6\ell_8^{(8)}
\]
already recorded in \cite[Example~37]{ShareshianSundaram}.
Thus $B_8(2)=-4$, and
\Cref{thm:effective-lifts} shows that $\sigma_8(R_8^{\sq})$ is not
effective.  Nevertheless, every Schur coefficient of $R_8^{\sq}$ is
positive by \Cref{thm:ramanujan-square}.  Taking Schur coefficients
and using \eqref{eq:KW} gives
\[
 2a_{\lambda,2}<3a_{\lambda,4}+3a_{\lambda,0}
 \qquad(\lambda\vdash8),
\]
where $a_{\lambda,0}=a_{\lambda,8}$.
\end{example}

\begin{example}\label{ex:degree-nine-regions}
\label{key:example-n9-beyond-foulkes-cone}
In degree $9$, suppress the fixed coordinate $x_1=1$ and use
$(u,v)=(x_3,x_9)$.  Then
\begin{equation}\label{eq:degree-nine-regions}
 \begin{aligned}
 \Delta_9&=\operatorname{conv}\{(0,0),(1,0),(1,1)\},\\
 H_9&=[0,1]^2,\\
 P_9^Q&=[-2,5]\times[0,1].
 \end{aligned}
\end{equation}
The first two descriptions follow from the definition of $\Delta_9$
and \Cref{thm:boolean-core}.  For the third, write
$F_9(u,v)=Q_{9,1}+uQ_{9,3}+vQ_{9,9}$, as in
\eqref{eq:degree-nine-function}.  The coefficient formulas in
\eqref{eq:degree-nine-coefficients} are obtained independently from
the character formulas in \Cref{subsec:small-character-formulas}.
The coefficients indexed by $(9)$, $(8,1)$, $(5,4)$, and $(3^3)$
force $0\le v\le1$ and $-2\le u\le5$.  Conversely, evaluating each
affine expression at the four vertices of this rectangle verifies its
nonnegativity throughout the rectangle.  This proves
the last equality in \eqref{eq:degree-nine-regions}.

The following functions illustrate the two strict inclusions.
Let $\Xi_9\coloneqq\sum_{d\mid9}p_d^{9/d}$.  The expansions
\[
 \begin{aligned}
 \Xi_9&=3\ell_9^{(1)}+3\ell_9^{(3)}+3\ell_9^{(9)},\\
 f_9^{\{1,9\}}&=\ell_9^{(1)}-\ell_9^{(3)}+\ell_9^{(9)},\\
 R_9^{\sq}&=-6\ell_9^{(1)}+10\ell_9^{(3)}+5\ell_9^{(9)}
 \end{aligned}
\]
follow, respectively, from \cite[Theorem~18]{ShareshianSundaram},
\eqref{eq:mobius-Q-Foulkes}, and \eqref{eq:ramanujan-B-coordinate}.
The third was recorded in \cite[Example~37]{ShareshianSundaram}.
Their normalized
$Q$-coordinates are
\[
 \begin{aligned}
 \Xi_9/9&=F_9(2/3,1/3),&&(2/3,1/3)\in\Delta_9,\\
 f_9^{\{1,9\}}&=F_9(0,1),&&(0,1)\in H_9\setminus\Delta_9,\\
 \widetilde R_9^{\sq}&=F_9(5/3,5/9),
 &&(5/3,5/9)\in P_9^Q\setminus H_9.
 \end{aligned}
\]
Thus both inclusions $\Delta_9\subsetneq H_9\subsetneq P_9^Q$ are
strict.  The negative Foulkes coefficient $B_9(1)=-6$ also shows that
$\sigma_9(R_9^{\sq})$ is not effective, whereas every Schur
coefficient of $R_9^{\sq}$ is positive.  By \eqref{eq:KW}, the last
expansion yields
\[
 6a_{\lambda,1}<10a_{\lambda,3}+5a_{\lambda,0}
 \qquad(\lambda\vdash9),
\]
where $a_{\lambda,0}=a_{\lambda,9}$.
\end{example}

\section{Facets defined by the trivial and sign coefficients}\label{sec:boundary}

The Boolean region and the Ramanujan-square points lie in the same
normalized slice.  We compare $\Delta_n,H_n,P_n^Q$ under the coefficient
maps $\pi_n$ and $\pi_n^Q$ defined in \eqref{eq:intro-endpoint-map}.
These maps record the two coefficients that already determine the
exceptional cases in both positivity results.
On the normalized slice, the four distinguished coefficients are
$\tau_n$, $1-\tau_n$, $1-\eps_n$, and $\eps_n$.  Their nonnegativity is
equivalent to $0\le\tau_n\le1$ and $0\le\eps_n\le1$.
We prove the facet assertions by applying \Cref{thm:criterion}
to weights supported on two or three divisors.

Put
\[
 \mathcal S\coloneqq\set{(t,t):0\le t\le1},
 \qquad
 \mathcal T\coloneqq\set{(u,v):u\ge0,\ v\ge0,\ u+v\le1}.
\]

\begin{theorem}\label{thm:endpoint-facets}
Let $n\ge18$.  If $n$ is odd, then
\[
 \pi_n^Q(\Delta_n)=\pi_n^Q(H_n)=\pi_n^Q(P_n^Q)=\mathcal S.
\]
If $n$ is even, then
\[
 \pi_n^Q(\Delta_n)=\pi_n^Q(H_n)=\mathcal T,
 \qquad
 \pi_n^Q(P_n^Q)=[0,1]^2.
\]
For even $n$, each of the four hyperplanes
\[
 \tau_n=0,\qquad\tau_n=1,\qquad\eps_n=0,\qquad\eps_n=1
\]
defines a facet of $P_n$.  Equivalently, under the $Q$-coordinate
identification, the corresponding hyperplane defines a facet of $P_n^Q$.
For odd $n$, the two distinct hyperplanes $\tau_n=0$ and $\tau_n=1$
define facets; the corresponding $\eps_n$-hyperplanes coincide with them.
\end{theorem}

\begin{proof}
Every point of $P_n$ satisfies $0\le\tau_n,\eps_n\le1$ by
\Cref{lem:endpoint-coefficients}.  If $n$ is odd, then $n-n/d$ is even for
every $d\mid n$, so $\eps_n=\tau_n$ on $\Rspace_{n,\mathbb R}$.  For
$0\le t\le1$, define a weight supported on $\set{1,n}$ by
\[
 w_t(1)\coloneqq 1,
 \qquad
 w_t(n)\coloneqq nt-1.
\]
Then $\tau_n(w_t)=\eps_n(w_t)=t$.  Since the support of $w_t$ among divisors
$d>1$ is contained in $\set{n}$, one has
$A_{n,r}(w_t)=0$ for every $2\le r<n/2$ and $E_n(w_t)=0$.  Hence
\Cref{thm:criterion} gives $F_n(w_t)\in P_n$, and
$\pi_n^Q(P_n^Q)=\mathcal S$.  The vertices
$\one_{\D(1)}$ and $\one_{\D(n)}$ of $\Delta_n$ map to the endpoints of
$\mathcal S$, so the two remaining equalities follow from
\eqref{eq:intro-nested-regions}.

Suppose that $n$ is even.  Formula \eqref{eq:endpoint-even-Q} gives
\[
 \pi_n^Q(x)=(x_n,x_{n/2}-x_n).
\]
The description of $H_n$ in \Cref{thm:boolean-core} yields
$\pi_n^Q(H_n)=\mathcal T$.  The vertices
$\one_{\D(1)}$, $\one_{\D(n/2)}$, and $\one_{\D(n)}$ of $\Delta_n$ map
to $(0,0)$, $(0,1)$, and $(1,0)$, respectively, while every vertex of
$\Delta_n$ maps to one of these three points.  Thus
$\pi_n^Q(\Delta_n)=\mathcal T$.

For $(u,v)\in[0,1]^2$, define a weight supported on
$\set{1,n/2,n}$ by
\begin{equation}\label{eq:wuv}
 w_{u,v}(1)\coloneqq 1,
 \qquad
 w_{u,v}(n)\coloneqq\frac n2(u-v),
 \qquad
 w_{u,v}(n/2)\coloneqq\frac n2(u+v)-1.
\end{equation}
By \eqref{eq:tau-epsilon},
$\tau_n(w_{u,v})=u$ and $\eps_n(w_{u,v})=v$.
Since $n/2>r$, one has $A_{n,r}(w_{u,v})=0$ for $2\le r<n/2$.
The only possible nonzero term in $E_n(w_{u,v})$ has $d=n/2$.
Using $|\frac n2(u+v)-1|\le n-1<n$, we obtain
\[
 \begin{aligned}
 E_n(w_{u,v})
 &=e^{1/24}X_n^{-2/3}
   \frac{|\frac n2(u+v)-1|}{\sqrt{n/2}}\\
 &\le e^{1/24}\sqrt{2n}\,X_n^{-2/3}
 \le\mathcal E_n^{\sq}<1.
 \end{aligned}
\]
Thus \Cref{thm:criterion} gives $F_n(w_{u,v})\in P_n$, proving
$\pi_n^Q(P_n^Q)=[0,1]^2$.

For $(u,v)$ in the relative interior of any side of the square,
the Schur coefficient corresponding to that side vanishes at
$F_n(w_{u,v})$, and every other Schur coefficient is positive.
Since only finitely many Schur inequalities define $P_n$, all the remaining
inequalities stay strict in a relative neighborhood of this point on
the hyperplane defined by the vanishing coefficient within the normalized
affine space $\set{F\in\Rspace_{n,\mathbb R}:[p_1^n]F=1/n}$ of
\eqref{eq:normalized-base}.  Since $P_n$ is full-dimensional in this affine
space, the resulting face has codimension one and is therefore a facet.
In odd degree, the two distinguished coefficients that
vanish at either $F_n(w_0)$ or $F_n(w_1)$ are equal as linear functionals
on $\Rspace_{n,\mathbb R}$ and hence impose a single supporting hyperplane.
All remaining Schur coefficients are positive, so the same argument proves
the odd-degree statement.
\end{proof}

\begin{remark}\label{rem:endpoint-degree-threshold}
The assertions about the image of $P_n^Q$ and the facets of $P_n$ and
$P_n^Q$ are restricted to $n\ge18$ because their proofs use the uniform
estimates in \Cref{thm:criterion}.  The sharpness of this threshold is not
addressed here.  We give no general description of the facets in smaller
degrees, although \Cref{ex:degree-nine-regions} gives a complete
description of $P_9^Q$.
\end{remark}

By \Cref{thm:boolean-core} and \eqref{eq:intro-endpoint-map},
$\pi_8^Q(H_8)=\mathcal T$.  Since
\[
 \pi_8(\widetilde R_8^{\sq})=\left(\frac34,\frac34\right)\notin\mathcal T,
\]
the $Q$-coordinate vector of $\widetilde R_8^{\sq}$ does not belong to $H_8$.
In degree $9$, \Cref{ex:degree-nine-regions} gives
\[
 \pi_9(\widetilde R_9^{\sq})=\left(\frac59,\frac59\right)\in\mathcal S=\pi_9^Q(H_9).
\]
Nevertheless, the $Q$-coordinate vector of $\widetilde R_9^{\sq}$
does not belong to $H_9$.  Thus membership in the image of $H_9$
under $\pi_9^Q$ does not determine membership in $H_9$.  If
$n\equiv2\pmod4$, then \Cref{cor:rho-location,thm:endpoint-facets} place
$\widetilde R_n^{\sq}$ in the relative interior of the facet $\eps_n=0$ for every
$n\ge18$.

\section{Module realizations and major-index residue inequalities}
\label{sec:lifts-residue}

Integral Schur-positive functions in $\Rspace_n$ are Frobenius
characteristics of $S_n$-modules.  Their Foulkes coordinates determine whether they arise
by induction from $C_n$, while the Kra\'skiewicz--Weyman formula
expresses their Schur coefficients in terms of major-index residues.
The inequalities in
\Cref{cor:weighted-boolean-residue,cor:ramanujan-weighted-residue}
are consequences of this correspondence.  We now obtain module
realizations and more explicit residue inequalities for the Boolean
sums classified in \Cref{sec:boolean}.

By semisimplicity, finite-dimensional complex $S_n$-modules $A$ and $B$
admit an embedding $A\hookrightarrow B$ if and only if
$\ch B-\ch A$ is Schur-positive.

For $J\subseteq\D(n)$, define
\begin{equation}\label{eq:betaJ}
 \beta_J(e)\coloneqq\sum_{\substack{d\in J\\e\mid d}}\mu(d/e)
 \qquad(e\mid n).
\end{equation}
By \eqref{eq:mobius-Q-Foulkes},
\begin{equation}\label{eq:boolean-Foulkes-expansion}
 Q_{n,J}=\sum_{e\mid n}\beta_J(e)\ell_n^{(e)}.
\end{equation}
Together with \eqref{eq:canonical-lift}, this gives the lift determined by
the representatives $\{V_e:e\mid n\}$:
\[
 \sigma_n(Q_{n,J})
 =\sum_{e\mid n}\beta_J(e)[V_e]\in\Vlat_n.
\]
By \Cref{thm:effective-lifts,prop:fixed-foulkes}, this lift is effective if
and only if $J=\varnothing$ or $J=\D(r)$ for
some $r\mid n$; in the latter case, $\sigma_n(Q_{n,J})=[V_r]$.

\begin{corollary}\label{cor:boolean-module-realization}
For every nonempty admissible $Q$-support $J$, there exists an $S_n$-module
$M_J$ such that $\ch M_J=Q_{n,J}$ and $\dim M_J=(n-1)!$.  If $n\ge2$,
then $\operatorname{Res}_{S_{n-1}}^{S_n}M_J$ is the regular
representation.  Moreover, $Q_{n,J}$ is the Frobenius characteristic of
a representation induced from a finite-dimensional $C_n$-module if and
only if $J$ is a principal order ideal.
\end{corollary}

\begin{proof}
Integrality and \Cref{thm:boolean-classification} give an $S_n$-module
whose Frobenius characteristic is $Q_{n,J}$.  Since $Q_{n,J}$ has first
$Q$-coordinate $1$,
\eqref{eq:weight-dimension} gives its dimension, and
\Cref{prop:branching} gives the restriction formula when $n\ge2$.
The final assertion follows from the preceding lift characterization and
\Cref{thm:effective-lifts}.
\end{proof}

For example, the admissible $Q$-support $J=\{1,2,3\}\subseteq\D(6)$ gives
\[
 Q_{6,J}=-\ell_6^{(1)}+\ell_6^{(2)}+\ell_6^{(3)}.
\]
This function is Schur-positive, but its negative Foulkes coefficient
precludes induction from a finite-dimensional $C_6$-module.
Thus its $Q$-coordinate vector belongs to $H_6\setminus\Delta_6$.

For an order ideal, inclusion--exclusion gives an alternating sum of
Foulkes characteristics.

\begin{proposition}\label{prop:order-ideal}
Let $J$ be a nonempty order ideal in $\D(n)$ with maximal elements
$d_1,\ldots,d_m$.  For $\varnothing\ne A\subseteq[m]$, put
$g_A\coloneqq\gcd\{d_i:i\in A\}$.  Then
\begin{equation}\label{eq:order-ideal-IE}
 Q_{n,J}=\sum_{\varnothing\ne A\subseteq[m]}
 (-1)^{|A|+1}\ell_n^{(g_A)}.
\end{equation}
Consequently, for every $\lambda\vdash n$,
\begin{equation}\label{eq:order-ideal-major-index}
 \sum_{\substack{\varnothing\ne A\subseteq[m]\\|A|\text{ even}}}
 a_{\lambda,g_A}
 \le
 \sum_{\substack{\varnothing\ne A\subseteq[m]\\|A|\text{ odd}}}
 a_{\lambda,g_A},
\end{equation}
and there exists an $S_n$-module embedding
\begin{equation}\label{eq:order-ideal-embedding}
 \bigoplus_{\substack{\varnothing\ne A\subseteq[m]\\|A|\text{ even}}}
 \Lie_n^{(g_A)}
 \lhook\joinrel\longrightarrow
 \bigoplus_{\substack{\varnothing\ne A\subseteq[m]\\|A|\text{ odd}}}
 \Lie_n^{(g_A)}.
\end{equation}
\end{proposition}

\begin{proof}
One has
\[
 J=\D(d_1)\cup\cdots\cup\D(d_m),
 \qquad
 \bigcap_{i\in A}\D(d_i)=\D(g_A).
\]
Inclusion--exclusion for the indicator functions, followed by
\eqref{eq:zeta-Q-Foulkes}, proves \eqref{eq:order-ideal-IE}.
Every nonempty order ideal is admissible: it contains $1$, and in even degree
it contains $n/2$ whenever it contains $n$.  Hence $Q_{n,J}$ is
Schur-positive by \Cref{thm:boolean-classification}.  Taking Schur
coefficients in \eqref{eq:order-ideal-IE} and using \eqref{eq:KW} gives
\eqref{eq:order-ideal-major-index}.  The difference between the Frobenius
characteristics of the target and source in \eqref{eq:order-ideal-embedding}
is $Q_{n,J}$, so semisimplicity gives the embedding.
\end{proof}

\begin{corollary}\label{cor:two-generated}
If $a,b\mid n$ and $g\coloneqq\gcd(a,b)$, then
\[
 \ell_n^{(a)}+\ell_n^{(b)}-\ell_n^{(g)}
\]
is Schur-positive.  Equivalently,
\begin{equation}\label{eq:two-generated-major-index}
 a_{\lambda,g}\le a_{\lambda,a}+a_{\lambda,b}
 \qquad(\lambda\vdash n),
\end{equation}
and there exists an $S_n$-module embedding
\[
 \Lie_n^{(g)}\lhook\joinrel\longrightarrow
 \Lie_n^{(a)}\oplus\Lie_n^{(b)}.
\]
\end{corollary}

\begin{proof}
If $a$ and $b$ are comparable, then $g$ is the smaller divisor, the
symmetric function is the Foulkes characteristic indexed by the larger, and
the source is a direct summand of the target.  Otherwise $a$ and
$b$ are the maximal elements of $J=\D(a)\cup\D(b)$, so
\Cref{prop:order-ideal} applies.
\end{proof}

\begin{remark}
The proofs establish existence through choices of irreducible
decompositions; they do not specify embeddings or injections between the
tableau sets counted in
\eqref{eq:order-ideal-major-index} and \eqref{eq:two-generated-major-index}.
We return to this issue in \Cref{sec:further-questions}.
\end{remark}

\subsection{Reduced fake degrees and two-term \texorpdfstring{$Q$}{Q}-sums}
\label{subsec:fake-degrees}
For $\lambda\vdash n$, its fake-degree polynomial is
\[
 \operatorname{FD}_\lambda(z)
 \coloneqq\sum_{U\in\operatorname{SYT}(\lambda)}
   z^{\maj(U)}
 =\sum_{j\ge0}b_{\lambda,j}z^j.
\]
Reduction modulo $z^n-1$ gives the unique representative of degree less than
$n$,
\begin{equation}\label{key:fake-degree-residue-sums}
 \overline{\operatorname{FD}}_{\lambda,n}(z)\coloneqq
 \sum_{r=0}^{n-1}a_{\lambda,r}z^r,
 \qquad
 a_{\lambda,r}
 =\sum_{\substack{j\ge0\\j\equiv r\pmod n}}b_{\lambda,j}.
\end{equation}
We call this the \emph{reduced fake-degree polynomial of $\lambda$ modulo
$z^n-1$}.  As in \eqref{eq:KW}, the second subscript of $a_{\lambda,r}$ is
read modulo $n$; in particular, $a_{\lambda,n}=a_{\lambda,0}$.

\begin{remark}
For $j\in\mathbb Z$, reduction modulo $z^n-1$ gives
\[
 \operatorname{FD}_\lambda(\zeta_n^j)
 =\sum_{r=0}^{n-1}a_{\lambda,r}\zeta_n^{jr}.
\]
By the Kra\'skiewicz--Weyman formula~\eqref{eq:KW} and Frobenius reciprocity,
$a_{\lambda,r}$ is the multiplicity of $V_r$ in
$\operatorname{Res}_{C_n}^{S_n}S^\lambda$.  Hence the displayed sum equals
$\chi^\lambda(\mathsf c^j)$.  Setting $j=n/d$, we obtain
\[
 \operatorname{FD}_\lambda(\zeta_d)
 =\chi^\lambda_{(d^{n/d})}
 \qquad(d\mid n).
\]
Consequently, for every $T\subseteq\mathbb N$,
\[
 \sum_{d\mid n}\psi^T(d)\operatorname{FD}_\lambda(\zeta_d)
 =n[s_\lambda]f_n^T\in n\mathbb Z.
\]
The divisibility by $n$ follows from \Cref{prop:Q-basis}.
\end{remark}

Recall from \eqref{eq:gamma-major-index-definition} that
$\gamma_{\lambda,d}=[s_\lambda]Q_{n,d}$ is obtained by M\"obius inversion
from the residue multiplicities $a_{\lambda,e}$.
Using the $Q$-basis expansion \eqref{eq:fT-Q-sum} and the
Foulkes expansion \eqref{eq:weighted-Foulkes-expansion}, we also obtain
\begin{equation}\label{key:equation-major-index-mobius-transform}
 [s_\lambda]f_n^T
 =\sum_{\substack{d\mid n\\d\in T}}\gamma_{\lambda,d}
 =\sum_{e\mid n}\beta_{n,T}(e)a_{\lambda,e}.
\end{equation}
For $T=\{1,k\}$ or $T=\{k^j:j\ge0\}$ with $k\ge2$,
\Cref{cor:one-k-classification,cor:powers-classification}
determine exactly when the sums in
\eqref{key:equation-major-index-mobius-transform}
are nonnegative for every $\lambda\vdash n$.

We next make the residue inequalities obtained from
\Cref{cor:one-k-classification} explicit.
Let $k=q_1^{a_1}\cdots q_r^{a_r}>1$, where the $q_i$ are distinct
primes, and put $k_I\coloneqq k/\prod_{i\in I}q_i$ for $I\subseteq[r]$.
If $k\mid n$, then
\begin{equation}\label{key:alternating-lie-expansion-one-k}
 \ell_{n/k}^{(1)}[p_k]
 =\sum_{I\subseteq[r]}(-1)^{|I|}\ell_n^{(k_I)}.
\end{equation}
This follows from \eqref{eq:mobius-Q-Foulkes} and
$Q_{n,k}=\ell_{n/k}^{(1)}[p_k]$, since, among the divisors
$e\mid k$, the nonzero values of $\mu(k/e)$ occur precisely
at $e=k_I$ and equal $(-1)^{|I|}$.  Since
$f_n^{\{1,k\}}=Q_{n,1}+Q_{n,k}$ and $Q_{n,1}=\ell_n^{(1)}$, it also gives
\[
 f_n^{\{1,k\}}=\ell_n^{(1)}+
 \sum_{I\subseteq[r]}(-1)^{|I|}\ell_n^{(k_I)}.
\]

\begin{corollary}\label{key:corollary-major-index-inequalities}
With the preceding notation, suppose that $k\mid n$ and that either
$n\ne k$, $k=2$, or $k$ is odd.
Then, for every $\lambda\vdash n$,
\[
 \sum_{\substack{I\subseteq[r]\\|I|\text{ odd}}}a_{\lambda,k_I}
 \le
 a_{\lambda,1}+
 \sum_{\substack{I\subseteq[r]\\|I|\text{ even}}}a_{\lambda,k_I}.
\]
\end{corollary}

\begin{proof}
By \Cref{cor:one-k-classification}, $f_n^{\{1,k\}}$ is
Schur-positive.  Taking Schur coefficients in its preceding
expansion, applying \eqref{eq:KW}, and separating the terms
according to the parity of $|I|$ gives the inequality.
\end{proof}

Under the same hypotheses, when $k=q^s$ with $q$ prime and $s\ge2$,
the inequality becomes
\[
 a_{\lambda,q^{s-1}}\le a_{\lambda,1}+a_{\lambda,q^s}.
\]
Equivalently,
\[
 \sum_{\substack{j\ge0\\j\equiv q^{s-1}\pmod n}}b_{\lambda,j}
 \le
 \sum_{\substack{j\ge0\\j\equiv1\pmod n}}b_{\lambda,j}
 +
 \sum_{\substack{j\ge0\\j\equiv q^s\pmod n}}b_{\lambda,j}.
\]
\subsection{Strictness and equality for composite divisors}
For a divisor $k>1$ of $n$ and $\lambda\vdash n$, define the residue difference
\begin{equation}\label{eq:residue-difference}
 D_\lambda(k)\coloneqq a_{\lambda,1}
       +\sum_{d\mid k}\mu(k/d)a_{\lambda,d}.
\end{equation}
By \eqref{eq:mobius-Q-Foulkes} and \eqref{eq:KW},
$D_\lambda(k)=[s_\lambda]f_n^{\{1,k\}}$.
For $k=q^s$, with $q$ prime and $s\ge2$, this is
$D_\lambda(k)=a_{\lambda,1}+a_{\lambda,q^s}-a_{\lambda,q^{s-1}}$.

We use the following weight formulas to evaluate $D_\lambda(k)$ at
the four distinguished partitions and estimate it for balanced partitions.
Put
\[
 \psi_k\coloneqq\psi^{\{1,k\}},\qquad
 \psi_k(d)=
 \begin{cases}
  \mu(d),&k\nmid d,\\
  \mu(d)+k\mu(d/k),&k\mid d
 \end{cases}
 \quad(d\ge1).
\]
\begin{lemma}\label{key:lemma-weights-for-one-k}
Let $k>1$ and $n\ge1$.
\begin{enumerate}[label=\textup{(\arabic*)}]
\item If $k\mid n$ and $n>k$, then
\[
 \tau_n(\psi_k)=0,
 \qquad
 \eps_n(\psi_k)=\one_{n=2k}.
\]
\item $|\psi_k(2)|=1$ and $|\psi_k(d)|\le2d$ for $d\ge3$.  If
$k$ is not squarefree, then $|\psi_k(d)|\le d$ for every $d$.
\end{enumerate}
\end{lemma}

\begin{proof}
Since $n>k$, one has $n\notin\{1,k\}$, whereas
$n/2\in\{1,k\}$ exactly when $n=2k$.  Part~\textup{(1)} follows from
\cref{lem:boolean-endpoints}.

For part~\textup{(2)}, the formula for $\psi_k$ gives
$|\psi_k(2)|=1$.  If $k\nmid d$, then
$|\psi_k(d)|\le1$; if $k\mid d$, then
$|\psi_k(d)|\le k+1\le2d$.  For the final assertion, suppose that $k$ is not squarefree.
If $k\nmid d$, then $|\psi_k(d)|=|\mu(d)|\le1\le d$.
If $k\mid d$, some prime square divides $d$, so $\mu(d)=0$ and
$|\psi_k(d)|=k|\mu(d/k)|\le d$.
\end{proof}

\begin{corollary}\label{cor:composite-equality}\label{key:corollary-strict-prime-power-inequality}
Let $k$ be composite, $k\mid n$, and $n>k$.  For every $\lambda\vdash n$
outside $(n),(n-1,1),(2,1^{n-2}),(1^n)$, one has
$D_\lambda(k)\ge1$.  At these four partitions its values are,
respectively,
\[
 0,\qquad1,\qquad1-\one_{n=2k},\qquad\one_{n=2k}.
\]
Consequently, $D_\lambda(k)=0$ precisely for $\lambda=(n)$, together with
$\lambda=(2,1^{n-2})$ when $n=2k$ and $\lambda=(1^n)$ when $n\ne2k$.
If $n\ge18$ and $\lambda\in\Bpart_n$, then
\begin{equation}\label{key:quantitative-prime-power-residue}
 \left|D_\lambda(k)-\frac{f^\lambda}{n}\right|
 \le\mathcal G_n(2)\frac{f^\lambda}{n}.
\end{equation}
When $k$ is not squarefree, $\mathcal G_n(2)$ may be replaced by
$\mathcal G_n(1)$.
\end{corollary}
\begin{proof}
Since $k\mid n$, $n>k$, and $k$ is composite, $n$ is composite
and $n\ge2k\ge8$.  Applying \Cref{cor:boolean-strictness}
for $n\ge18$ and \Cref{prop:boolean-low} for $n<18$ to
$J=\{1,k\}$ gives strict positivity outside the four
distinguished partitions.  Integrality then gives
$D_\lambda(k)\ge1$.
The four values follow from
\Cref{lem:endpoint-coefficients,key:lemma-weights-for-one-k}.
Finally apply
Corollary~\ref{key:corollary-quantitative-balanced}, first with
$\kappa=2$ and then with $\kappa=1$ for nonsquarefree $k$.
\end{proof}
\subsection{Ungraded module reformulations}
The residue inequalities also admit equivalent ungraded
representation-theoretic formulations.
Let
\[
\operatorname{Coinv}_n\coloneqq
 \mathbb C[x_1,\ldots,x_n]\big/
 \left\langle\mathbb C[x_1,\ldots,x_n]^{S_n}_+\right\rangle
\]
be the type-$A$ coinvariant algebra, with $S_n$ acting by permuting the
variables.  Here $\mathbb C[x_1,\ldots,x_n]^{S_n}_+$ denotes the
positive-degree part of the invariant ring, and
$\operatorname{Coinv}_n[j]$ denotes the homogeneous component of degree $j$.
For $a\in\mathbb Z$, define the component supported in degrees congruent to
$a$ modulo $n$ by
\[
 \operatorname{Coinv}_n^{[a]}
 \coloneqq\bigoplus_{\substack{j\ge0\\j\equiv a\pmod n}}
 \operatorname{Coinv}_n[j].
\]
Combining the Lusztig--Stanley graded multiplicity formula for the coinvariant
algebra \cite[Proposition~4.11]{StanleyInvariants}, in the form stated in
\cite[Theorem~2.4]{Swanson}, with the Kra\'skiewicz--Weyman
cyclic-induction formula \cite[Theorem~1]{KW} gives
\begin{equation}\label{key:coinvariant-residue-character}
 \ch\operatorname{Coinv}_n^{[a]}
 =\sum_{\lambda\vdash n}a_{\lambda,a}s_\lambda
 =\ell_n^{(a)}.
\end{equation}

\begin{corollary}\label{key:corollary-coinvariant-embedding}
Let $k=q_1^{a_1}\cdots q_r^{a_r}>1$, where the $q_i$ are distinct primes,
and let $n$ be a multiple of $k$.
Assume that either $n\ne k$, $k=2$, or $k$ is odd.  For $I\subseteq[r]$,
let $k_I\coloneqq k/\prod_{i\in I}q_i$.  There exists an
ungraded $S_n$-module embedding
\[
 \bigoplus_{\substack{I\subseteq[r]\\|I|\text{ odd}}}
 \operatorname{Coinv}_n^{[k_I]}
 \hookrightarrow
 \operatorname{Coinv}_n^{[1]}\oplus
 \bigoplus_{\substack{I\subseteq[r]\\|I|\text{ even}}}
 \operatorname{Coinv}_n^{[k_I]}.
\]
\end{corollary}

\begin{proof}
By \eqref{key:coinvariant-residue-character} and
\eqref{key:alternating-lie-expansion-one-k}, the Frobenius
characteristic of the target minus that of the source is
$f_n^{\{1,k\}}$, which is Schur-positive by
\Cref{cor:one-k-classification}.  Semisimplicity gives the embedding.
\end{proof}

\begin{remark}\label{key:remark-induced-embedding}
The proof compares multiplicities after forgetting the grading.  It does
not specify an embedding or prove that one can be chosen to preserve degrees.
Equation \eqref{key:coinvariant-residue-character} shows that, for every
$a\in\mathbb Z$,
$\operatorname{Coinv}_n^{[a]}\cong\Lie_n^{(a)}$ as ungraded
$S_n$-modules.  Replacing the residue components in
\cref{key:corollary-coinvariant-embedding} by these isomorphic modules gives
the equivalent embedding
\[
 \bigoplus_{\substack{I\subseteq[r]\\|I|\text{ odd}}}
 \Lie_n^{(k_I)}
 \hookrightarrow
 \Lie_n^{(1)}\oplus
 \bigoplus_{\substack{I\subseteq[r]\\|I|\text{ even}}}
 \Lie_n^{(k_I)}.
\]
In the prime-power case $k=q^s$, with $q$ prime and $s\ge2$, this becomes
\[
 \Lie_n^{(q^{s-1})}
 \hookrightarrow
 \Lie_n^{(q^s)}\oplus\Lie_n^{(1)}.
\]
\end{remark}

\section{Further questions}\label{sec:further-questions}

By \Cref{thm:ramanujan-square,cor:ramanujan-Q-lattice}, $R_n^{\sq}$ has
nonnegative integral Schur coefficients and hence is the Frobenius
characteristic of an $S_n$-module.

\begin{question}\label{ques:explicit-module}
For each $n\ge1$, can one construct such a module
with an explicitly described basis and formulas for the action of
the adjacent transpositions $(i\,\,i+1)$, $1\le i<n$, without first
computing the Schur coefficients of $R_n^{\sq}$?
\end{question}

Recall that $\operatorname{SYT}_r(\lambda)$ denotes the set of standard
Young tableaux of shape $\lambda$ whose major index is congruent to $r$
modulo $n$.
\begin{question}\label{ques:tableau-injection}
Does the major-index inequality \eqref{eq:two-generated-major-index} in
\Cref{cor:two-generated} admit a direct tableau-theoretic proof?  More
precisely, for every $n$, every $\lambda\vdash n$, and every pair of
incomparable divisors $a,b\mid n$, can one define, by explicit operations on
the input tableau, a map
\[
 \operatorname{SYT}_{\gcd(a,b)}(\lambda)
 \lhook\joinrel\longrightarrow
 \operatorname{SYT}_a(\lambda)\sqcup\operatorname{SYT}_b(\lambda)?
\]
Can one then verify directly that each output has the required
major-index residue and that the map is injective, without using
the established cardinality inequality or an irreducible decomposition?
\end{question}

\begin{question}\label{ques:polytope-geometry}
The following aspects of the full geometry of $P_n^Q$ remain open.
\begin{enumerate}[label=\textup{(\alph*)}]
\item Which distinct supporting hyperplanes arising from the Schur
inequalities in \eqref{eq:PnQ} define facets of $P_n^Q$?  Here coincident
supporting hyperplanes are to be identified; \Cref{thm:endpoint-facets}
gives such facets for $n\ge18$.
\item For which $n$ is $P_n^Q$ a lattice polytope?
\item In those degrees, does $P_n^Q$ have the integer decomposition
property with respect to $\mathbb Z^{\D(n)}$?
\end{enumerate}
\end{question}

\appendix
\section{Numerical estimates for divisor weights}\label{app:estimates}

We verify conditions \textup{(ii)} and \textup{(iii)} of
\Cref{thm:criterion} for $n\ge18$ by reducing the required inequalities
to finitely many exact bounds.  We first recall three estimates from
\cite[Sections~5 and~6]{Hou}, then apply them to the Boolean,
Ramanujan-square, and linear weights.  All tabulated values and rational
inequalities are exact.

\subsection{Monotonicity estimates}\label{subsec:appendix-monotonicity}

\begin{lemma}[{\cite[Lemma~5.3]{Hou}}]
\label{lem:short-cycle-monotonicity}
Let $r\ge2$, $2\le d\le r$, and
$1\le q\le\lfloor r/d\rfloor$.  Then
\[
 \frac{\binom{\lfloor n/d\rfloor}{q}}{g_r(n)}
\]
is nonincreasing as a function of the integer $n\ge2r+1$.
\end{lemma}

\begin{lemma}[{\cite[equation~(32)]{Hou}}]
\label{lem:central-binomial}
For every $r\ge5$,
\[
 \binom{2r+1}{r}\ge\frac{4^r}{\sqrt r}.
\]
\end{lemma}

\begin{lemma}[{\cite[Lemma~6.2]{Hou}}]
\label{lem:X-ratio-appendix}
For $n\ge18$,
\[
 \frac{\Ldim_{n+1}}{\Ldim_n}\ge\frac95
 \qquad\text{and consequently}\qquad
 \frac{X_{n+1}}{X_n}\ge\frac95\sqrt{\frac n{n+1}}.
\]
\end{lemma}

\begin{lemma}\label{lem:common-long-row-tail}
Let $r\ge5$, and let $M(2),\ldots,M(r)$ be nonnegative real numbers with
$M(2)=1$.  Put
\[
 A^M_r\coloneqq\sum_{d=2}^r M(d)
 \sum_{q=1}^{\lfloor r/d\rfloor}
 \binom{\lfloor(2r+1)/d\rfloor}{q}.
\]
Then
\begin{equation}\label{eq:common-long-row-tail}
 \frac{A^M_r}{g_r(2r+1)}
 <\frac{(r+2)\sqrt r}{2}
 \left(
 2^{-r}+\left(\sum_{d=3}^rM(d)\right)2^{(1-4r)/3}
 \right).
\end{equation}
\end{lemma}

\begin{proof}
As in \cite[proof of Proposition~5.4, equations~(31)--(33)]{Hou},
the inner binomial sum is less than $2^r$ for $d=2$ and at most
$2^{\lfloor(2r+1)/d\rfloor}\le2^{(2r+1)/3}$ for $d\ge3$.  Thus
\[
 A^M_r
 <2^r+\left(\sum_{d=3}^rM(d)\right)2^{(2r+1)/3}.
\]
Divide by
$g_r(2r+1)=\frac2{r+2}\binom{2r+1}{r}$ and apply
\Cref{lem:central-binomial}.
\end{proof}

\begin{lemma}\label{lem:ramanujan-tail-decrease}
For $r\ge12$, put
\[
 \Sigma_r\coloneqq\sum_{d=3}^rd^2,
 \qquad
 W_r\coloneqq\frac{(r+2)\sqrt r}{2}
 \left(2^{-r}+\Sigma_r2^{(1-4r)/3}\right).
\]
Then
\[
 W_r\le W_{12}<\frac{10731}{32768}<\frac13.
\]
\end{lemma}

\begin{proof}
The bounds $\Sigma_{12}=645$, $\sqrt{12}<7/2$, and $2^{2/3}>3/2$
give
\[
 W_{12}<\frac{49}{2}\left(\frac1{4096}+\frac{645}{49152}\right)
 =\frac{10731}{32768}.
\]
For $r\ge12$, the successive ratio of the first summand is
\[
 \frac{r+3}{2(r+2)}\sqrt{\frac{r+1}{r}}
 \le\frac12\cdot\frac{15}{14}\cdot\frac{25}{24}<1.
\]
The inequality $\Sigma_r\ge3(r+1)^2$ holds at $r=12$ and is
preserved when $r$ increases by one, since
$(r+1)^2-3(2r+3)>0$ for $r\ge12$.
Thus $\Sigma_{r+1}/\Sigma_r=1+(r+1)^2/\Sigma_r\le4/3$.
The successive ratio of the second summand is
\[
 2^{-4/3}\frac{r+3}{r+2}\sqrt{\frac{r+1}{r}}
 \frac{\Sigma_{r+1}}{\Sigma_r}
 <\frac25\cdot\frac{15}{14}\cdot\frac{25}{24}\cdot\frac43
 =\frac{25}{42}<1.
\]
\end{proof}

\begin{lemma}\label{lem:balanced-factor-decrease}
For every $0\le\alpha\le1$, the sequence
\[
 n^\alpha X_n^{-2/3}
\]
is strictly decreasing for $n\ge18$.  The sequence $X_n^{-1/2}$ is also
strictly decreasing in this range.
\end{lemma}

\begin{proof}
Following \cite[Lemmas~6.2 and~6.3]{Hou}, apply
\Cref{lem:X-ratio-appendix} to obtain
\[
 \frac{(n+1)^\alpha X_{n+1}^{-2/3}}
 {n^\alpha X_n^{-2/3}}
 \le
 \left(\frac59\right)^{2/3}
 \left(1+\frac1n\right)^{\alpha+1/3}
 \le
 \left(\frac59\right)^{2/3}
 \left(\frac{19}{18}\right)^{4/3}<1.
\]
It also gives $X_{n+1}>X_n$, proving the second assertion.
\end{proof}

\begin{lemma}\label{lem:balanced-base-cases}
The following estimates hold.
\begin{enumerate}[label=\textup{(\alph*)}]
\item
\[
 e^{1/24}\left(
 \frac{X_{20}^{-1/2}}{\sqrt2}
 +7\cdot20^{11/12}X_{20}^{-2/3}\right)
 <\frac{12356}{13455}<1.
\]
\item For every $J\subseteq\D(18)$ with $1\in J$,
\[
 E_{18}(\psi^J)<\frac{52}{295}<1,
\]
and for every $J\subseteq\D(19)$ with $1\in J$, one has
$E_{19}(\psi^J)=0$.
\item
\[
 \mathcal E_{18}^{\sq}<\frac{3196}{9085}<1.
\]
\end{enumerate}
\end{lemma}

\begin{proof}
Equation~\eqref{eq:Lambda-X} gives
$\Ldim_{18}=4862$ and $\Ldim_{20}=16796$.
We use $e^{1/24}<24/23$, which follows from
$\log(1+t)>t/(1+t)$ at $t=1/23$.
Since $\pi<22/7$, we have
\[
 X_n^2>\frac{7(\Ldim_n)^2}{44n}.
\]
The integer comparisons
$7(\Ldim_n)^2>44nL^2$ for $(n,L)=(18,457),(20,1490)$ therefore give
$X_{18}>457$ and $X_{20}>1490$.
For part~\textup{(a)}, use
$20^{11}<16^{12}$, $1490^2>130^3$, and $2\cdot1490>54^2$ to obtain
\[
 20^{11/12}<16,\qquad
 X_{20}^{2/3}>130,\qquad
 \frac1{\sqrt{2X_{20}}}<\frac1{54}.
\]
Thus the expression in part~\textup{(a)} is less than
\[
 \frac{24}{23}\left(\frac1{54}+\frac{56}{65}\right)
 =\frac{12356}{13455}.
\]

For part~\textup{(b)}, the absolute values of the Boolean weights at the proper
nontrivial divisors $2,3,6,9$ of $18$ are at most $1,3,7,9$,
respectively.
Since $64\cdot457^2>237^3$ and $2\cdot457>30^2$, the bounds
\[
 X_{18}^{2/3}>\frac{237}{4}>59,
 \qquad\frac1{\sqrt{2X_{18}}}<\frac1{30},
\]
give
\[
 E_{18}(\psi^J)
 <\frac{24}{23}\left(
 \frac1{30}+\frac{\sqrt3+7/\sqrt6+3}{59}\right)
 <\frac{24}{23}\left(\frac1{30}+\frac8{59}\right)
 =\frac{52}{295}.
\]
There is no proper nontrivial divisor of $19$, so $E_{19}(\psi^J)=0$.

The same bounds prove part~\textup{(c)}:
\[
 \mathcal E_{18}^{\sq}
 <\frac{24}{23}\left(\frac1{30}+18\frac4{237}\right)
 =\frac{3196}{9085}.
\]
\end{proof}

\subsection{Exact values used in the long-row estimates}
\label{subsec:finite-long-row-values}

\Cref{tab:quadratic-short} gives the values used in
\Cref{prop:ramanujan-short-cycle}.  The majorant
$A^{\sq}_{n_0,r}$ in \eqref{eq:Asqnr} sums over all integers
$3\le d\le r$, not just divisors of $n_0$.

\begin{table}[htbp]
\centering
\small
\begin{tabular}{ccrrc}
\toprule
$r$&$n_0$&$A^{\sq}_{n_0,r}$&$g_r(n_0)$&$A^{\sq}_{n_0,r}/g_r(n_0)$\\
\midrule
2&18&9&135&$1/15$\\
3&18&63&663&$21/221$\\
4&18&163&2244&$163/2244$\\
5&18&238&5508&$7/162$\\
6&18&565&9996&$565/9996$\\
7&18&663&13260&$1/20$\\
8&18&1013&11934&$1013/11934$\\
9&19&1355&16796&$1355/16796$\\
10&21&2439&58786&$2439/58786$\\
11&23&3067&208012&$3067/208012$\\
\bottomrule
\end{tabular}
\caption{Exact values entering the estimate
$A^{\sq}_{n_0,r}/g_r(n_0)$ in \Cref{prop:ramanujan-short-cycle}.}
\label{tab:quadratic-short}
\end{table}

\FloatBarrier
For example, the row with $(n_0,r)=(18,6)$ is obtained from
\[
 \begin{aligned}
 A^{\sq}_{18,6}
 &=\left(\binom91+\binom92+\binom93\right)
   +9\left(\binom61+\binom62\right)\\
 &\quad+16\binom41+25\binom31+36\binom31\\
 &=129+189+64+75+108=565,\\
 g_6(18)&=\binom{18}{6}-\binom{18}{5}=9996.
 \end{aligned}
\]

\subsection{Uniform bounds for the linear-weight error}
The following estimate completes the proof of
\Cref{cor:linear-criterion}.
\begin{proposition}
\label{key:proposition-G-numerics}
For each fixed $\kappa>0$, the sequence
$(\mathcal G_n(\kappa))_{n\ge18}$ is strictly decreasing.  Moreover,
\[
 \mathcal G_{18}(2)<1.
\]
Consequently,
\[
 \mathcal G_n(\kappa)<1\qquad(n\ge18,\ 0<\kappa\le2).
\]
\end{proposition}

\begin{proof}
By \Cref{lem:balanced-factor-decrease}, both terms defining
$\mathcal G_n(\kappa)$ are strictly decreasing in $n$ for $n\ge18$.
The proof of \Cref{lem:balanced-base-cases} gives
\[
 e^{1/24}<\frac{24}{23},\qquad
 X_{18}^{2/3}>\frac{237}{4},\qquad
 \frac1{\sqrt{2X_{18}}}<\frac1{30}.
\]
Since $256\cdot18^3<35^4$, one has $18^{3/4}<35/4$.  Hence
\[
 \mathcal G_{18}(2)
 <\frac{24}{23}\left(\frac1{30}+6\frac{35}{237}\right)
 =\frac{8716}{9085}<1.
\]
The last assertion follows because $\mathcal G_n(\kappa)$ increases
with $\kappa$.
\end{proof}
\section{The cases \texorpdfstring{$n<18$}{n<18}}\label{app:computations}

We prove the small-degree results needed for
\Cref{thm:boolean-classification,thm:ramanujan-square}.
We use the individual character estimates of \Cref{sec:criterion},
known Foulkes expansions, and explicit character values.
The formulas below also justify the coefficient calculations in
\Cref{ex:degree-nine-regions}.

\subsection{Character formulas}\label{subsec:small-character-formulas}

For $d\mid n$ and $j\in\mathbb Z$, put
\[
 H_d(j)\coloneqq
 \begin{cases}
  1/(j/d)!,&j\ge0\text{ and }d\mid j,\\
  0,&\text{otherwise}.
 \end{cases}
\]
If $\lambda$ has $\ell$ parts, the Jacobi--Trudi identity
(see \cite[Chapter~7]{StanleyEC2}) gives
\begin{equation}\label{eq:rectangular-determinant}
 \chi^\lambda_{(d^{n/d})}
 =(n/d)!\det\bigl(H_d(\lambda_i-i+j)\bigr)_{1\le i,j\le\ell}.
\end{equation}
Indeed, under the specialization $p_d=d$ and $p_j=0$ for $j\ne d$,
the identity
$\sum_{j\ge0}h_jz^j=\exp(\sum_{j\ge1}p_jz^j/j)$
shows that $h_j$ becomes $H_d(j)$.  The Frobenius formula shows that
$s_\lambda$ becomes $\chi^\lambda_{(d^{n/d})}/(n/d)!$.
In particular, $d=1$ in \eqref{eq:rectangular-determinant} gives
$f^\lambda$.  Formula \eqref{eq:Q-rectangular-expansion} then yields
\begin{equation}\label{eq:small-gamma-character}
 \gamma_{\lambda,d}
 =\frac dn\sum_{u\mid n/d}\mu(u)\chi^\lambda_{((du)^{n/(du)})}.
\end{equation}

We also use the Murnaghan--Nakayama rule
\cite[Theorem~4.10.2]{Sagan}.  For a rim hook $\lambda/\nu$, let
$\operatorname{ht}(\lambda/\nu)$ be one less than its number of rows.
For $\alpha=(\alpha_1,\ldots,\alpha_t)\vdash n$,
\begin{equation}\label{eq:MN-border-strip-appendix}
 \chi^\lambda_\alpha=
 \sum_{\substack{\nu\subseteq\lambda\\
       \lambda/\nu\text{ a rim hook of length }\alpha_1}}
 (-1)^{\operatorname{ht}(\lambda/\nu)}
 \chi^\nu_{(\alpha_2,\ldots,\alpha_t)},
 \qquad \chi^\varnothing_\varnothing=1.
\end{equation}

For $u,v\in\mathbb R$, define the degree-$9$ function
\begin{equation}\label{eq:degree-nine-function}
 \begin{aligned}
 F_9(u,v)&\coloneqq Q_{9,1}+uQ_{9,3}+vQ_{9,9}\\
 &=\frac19p_1^9+\frac{3u-1}{9}p_3^3
   +\left(v-\frac u3\right)p_9.
 \end{aligned}
\end{equation}
Thus
\[
 [s_\lambda]F_9(u,v)
 =\frac{f^\lambda-\chi^\lambda_{(3^3)}}9
 +\frac u3\bigl(\chi^\lambda_{(3^3)}-\chi^\lambda_{(9)}\bigr)
 +v\chi^\lambda_{(9)}.
\]
\begin{samepage}
Formula~\eqref{eq:rectangular-determinant} gives the following table.
Since the cycle types have only odd cycles, conjugate partitions give
identical entries; we list one representative of each pair
$\{\lambda,\lambda'\}$.
\begin{equation}\label{eq:degree-nine-coefficients}
\begin{array}{c|rrr|l}
\lambda&f^\lambda&\chi^\lambda_{(3^3)}&\chi^\lambda_{(9)}
 &[s_\lambda]F_9(u,v)\\ \hline
(9)&1&1&1&v\\
(8,1)&8&-1&-1&1-v\\
(7,2)&27&0&0&3\\
(7,1^2)&28&1&1&3+v\\
(6,3)&48&3&0&5+u\\
(6,2,1)&105&-3&0&12-u\\
(6,1^3)&56&2&-1&6+u-v\\
(5,4)&42&-3&0&5-u\\
(5,3,1)&162&0&0&18\\
(5,2,2)&120&3&0&13+u\\
(5,2,1^2)&189&0&0&21\\
(5,1^4)&70&-2&1&8-u+v\\
(4^2,1)&84&3&0&9+u\\
(4,3,2)&168&-3&0&19-u\\
(4,3,1^2)&216&0&0&24\\
 (3^3)&42&6&0&4+2u
\end{array}
\end{equation}
\end{samepage}

For example, the character value for $(5,4)$ can also be obtained from
\eqref{eq:MN-border-strip-appendix}.
The removable rim hooks of length $3$ leave $(5,1)$ and $(3,3)$,
with heights $0$ and $1$, respectively, and
\[
 \chi^{(5,1)}_{(3^2)}=\chi^{(2,1)}_{(3)}=-1,\qquad
 \chi^{(3,3)}_{(3^2)}
 =\chi^{(3)}_{(3)}-\chi^{(2,1)}_{(3)}=2.
\]
Hence $\chi^{(5,4)}_{(3^3)}=-3$.  Since
$f^{(5,4)}=\binom94-\binom93=42$ and $\chi^{(5,4)}_{(9)}=0$,
the corresponding coefficient is $5-u$.
In particular, the same calculation gives
$[s_{(5,4)}](p_1^9+4p_3^3)=30$.

\subsection{Boolean \texorpdfstring{$Q$}{Q}-sums}\label{subsec:boolean-low}

\begin{proposition}\label{prop:boolean-low}
Let $n<18$ be composite and $1\in J\subseteq\D(n)$.
Every Schur coefficient of $Q_{n,J}$ outside
\[
 (n),\quad(n-1,1),\quad(2,1^{n-2}),\quad(1^n)
\]
is nonnegative.  If $n\ge8$, every such coefficient is positive.
\end{proposition}

\begin{proof}
Write $x_d=\one_{d\in J}$, so that $x_1=1$ and
\[
 \psi^J(d)=\sum_{e\mid d}e\mu(d/e)x_e.
\]
We first treat long first rows and columns in degrees
$8,9,10,12,14,15,16$, using the following bounds for $|\psi^J(d)|$
at proper divisors $1<d<n$:
\begin{equation}\label{eq:small-boolean-weight-bounds}
\begin{array}{c|c|c}
 n&(d:1<d<n,\ d\mid n)&(|\psi^J(d)|\text{ bounded by})\\ \hline
8&(2,4)&(1,4)\\
9&(3)&(2)\\
10&(2,5)&(1,4)\\
12&(2,3,4,6)&(1,2,4,7)\\
14&(2,7)&(1,6)\\
15&(3,5)&(2,4)\\
16&(2,4,8)&(1,4,8)
\end{array}
\end{equation}
These bounds follow by substituting $x_e\in\{0,1\}$ in the preceding
divisor sum.  For example,
$\psi^J(4)=4x_4-2x_2$ and
$\psi^J(6)=1-2x_2-3x_3+6x_6$.
Substitution in \eqref{eq:Anr} gives the following upper bounds for
$A_{n,r}(\psi^J)/g_r(n)$; a dash means $r\ge n/2$.
\begin{equation}\label{eq:small-long-row-bounds}
\begin{array}{c|cccccc}
n&r=2&r=3&r=4&r=5&r=6&r=7\\ \hline
8&1/5&1/7&\text{--}&\text{--}&\text{--}&\text{--}\\
9&0&1/8&1/7&\text{--}&\text{--}&\text{--}\\
10&1/7&1/15&1/6&\text{--}&\text{--}&\text{--}\\
12&1/9&1/11&41/275&41/297&\text{--}&\text{--}\\
14&1/11&1/39&4/91&4/143&9/143&\text{--}\\
15&0&1/35&1/91&11/819&3/143&21/715\\
16&1/13&1/55&13/315&1/49&27/910&9/286
\end{array}
\end{equation}
For instance,
\[
 \begin{aligned}
 A_{12,4}(\psi^J)
 &\le\binom61+\binom62+2\binom41+4\binom31=41,\\
 g_4(12)&=\binom{12}4-\binom{12}3=275.
 \end{aligned}
\]
All entries are at most $1/5$.

For $\lambda=(n-r,\mu)$ with $2\le r<n/2$,
\eqref{eq:long-row-Fnw} and \Cref{lem:long-row-dimension} therefore give
\begin{equation}\label{eq:small-long-row-coefficient}
 n[s_\lambda]F_n(w)
 \ge\frac45f^\lambda-\delta_\mu|n\tau_n(w)-1|,
 \qquad
 w(d)=\psi^J(d)\ \text{or}\ (-1)^{n-n/d}\psi^J(d).
\end{equation}
Both weights have the same absolute values, and their trivial
coefficients are $\tau_n(\psi^J)$ and $\eps_n(\psi^J)$.
By \Cref{lem:boolean-endpoints}, their absolute values are at most $1$.
If $\mu\ne(1^r)$, the right-hand side is positive.  Otherwise it is at least
\[
 \frac45\binom{n-1}{2}-(n+1)
 =\frac{2n^2-11n-1}{5}>0
 \qquad(n\ge8).
\]
Conjugation thus proves the assertion for both long rows and long columns.

Next let $\lambda\in\Bpart_n$, with
$n\in\{9,10,12,14,15,16\}$.  Put
\begin{equation}\label{eq:small-theta}
 \theta_{n,d}\coloneqq
 d^{-1/2}e^{d/(12n)}X_n^{-(1-1/d)},\qquad
 \widehat E_n(J)\coloneqq
 \sum_{\substack{d\mid n\\1<d<n}}|\psi^J(d)|\theta_{n,d}.
\end{equation}
A balanced partition is not a hook, so $\chi^\lambda_{(n)}=0$.
By \Cref{lem:balanced-dimension,lem:normalized-FL},
\[
 n[s_\lambda]Q_{n,J}\ge f^\lambda(1-\widehat E_n(J)).
\]
The following bounds are sufficient:
\begin{equation}\label{eq:small-balanced-bounds}
\begin{array}{c|r|r|c|c}
n&\Ldim_n&X_n>&(\theta_{n,d}\text{ bounded above by})
 &\widehat E_n(J)<\\ \hline
9&126/5&3&(1/3)&2/3\\
10&42&5&(1/3,2/15)&13/15\\
12&132&15&(1/5,1/9,18/245,1/20)&8423/8820\\
14&429&40&(1/8,1/20)&17/40\\
15&6435/8&64&(1/20,1/20)&3/10\\
16&1430&140&(1/15,3/235,3/575)&12929/81075
\end{array}
\end{equation}
The divisor order in each tuple is that of
\eqref{eq:small-boolean-weight-bounds}.  To verify the bounds without
decimal approximations, use
\[
 X_n^2>\frac{7(\Ldim_n)^2}{44n},
 \qquad e^x<(1-x)^{-1}\quad(0<x<1).
\]
Thus $X_n>L$ and $\theta_{n,d}<a/b$ follow from the rational comparisons
\begin{equation}\label{eq:small-theta-integer-test}
 7(\Ldim_n)^2>44nL^2,\qquad
 b^{2d}(12n)^{2d}
 <a^{2d}d^d(12n-d)^{2d}L^{2d-2},
\end{equation}
respectively, for the values $L$ and $a/b$ in the table.
For example, $15^3>7^4$ gives
\[
 \theta_{12,4}
 <\frac12\frac{36}{35}\,15^{-3/4}
 <\frac12\frac{36}{35}\frac17=\frac{18}{245}.
\]
For the entry with $(n,d)=(10,5)$, the inequalities
$5^9>18^5$ and $1/\sqrt5<9/20$ give
\[
 \theta_{10,5}
 <\frac{24}{23}\frac9{20}\frac5{18}
 =\frac3{23}<\frac2{15}.
\]

The last column of \eqref{eq:small-balanced-bounds} follows from the
weight bounds and the bounds on $\theta_{n,d}$.  In degree $12$, retain
the dependence on $x_3$: put $b=x_3$, so that
\[
 |\psi^J(3)|=1+b,\qquad |\psi^J(6)|\le7-3b.
\]
Indeed, the possible values of $\psi^J(6)$ are $1,-1,7,5$ when $b=0$
and $-2,-4,4,2$ when $b=1$.  Consequently,
\[
 \widehat E_{12}(J)
 <\frac15+\frac{1+b}{9}
   +4\frac{18}{245}+\frac{7-3b}{20}
 =\frac{8423}{8820}-\frac{7b}{180}
 \le\frac{8423}{8820}<1.
\]
Every entry in the last column is less than $1$, proving positivity
for these balanced partitions.

It remains to consider the partitions outside the four distinguished
ones in degrees $4$ and $6$, and the balanced partitions in degree $8$.
Equations \eqref{eq:rectangular-determinant} and
\eqref{eq:small-gamma-character} give the following complete list,
up to the conjugations specified below.  The divisors in each tuple
are in increasing order.
\begin{equation}\label{eq:small-boolean-coefficients}
\begin{array}{c|c|c}
n&\lambda&(\gamma_{\lambda,d})_{d\mid n}\\ \hline
4&(2,2)&(0,1,0)\\ \hline
6&(4,2)&(1,1,0,0)\\
 &(4,1^2)&(2,-1,0,1)\\
 &(3,3)&(1,-1,1,0)\\
 &(3,2,1)&(3,0,-1,0)\\
 &(3,1^3)&(1,1,1,-1)\\
 &(2^3)&(0,1,1,0)\\
 &(2^2,1^2)&(2,-1,0,0)\\ \hline
8&(4,4)&(1,1,1,0)\\
 &(4,3,1)&(9,0,-1,0)\\
 &(4,2,2)&(6,2,0,0)\\
 &(4,2,1^2)&(12,-2,1,0)\\
 &(3,3,2)&(6,-2,1,0)
\end{array}
\end{equation}
Only determinants of order at most four are needed.  For example,
for $\lambda=(4,4)$ they give
\[
 \begin{aligned}
 f^\lambda
 &=8!\left(\frac1{(4!)^2}-\frac1{3!5!}\right)=14,
 &\chi^\lambda_{(2^4)}&=\frac{4!}{(2!)^2}=6,\\
 \chi^\lambda_{(4^2)}&=2!=2,
 &\chi^\lambda_{(8)}&=0.
 \end{aligned}
\]
Hence the corresponding tuple is
\[
 \left(\frac{14-6}{8},\frac{6-2}{4},\frac22,0\right)=(1,1,1,0).
\]
The balanced partitions of $8$ not displayed are the conjugates of
$(4,4)$, $(4,3,1)$, and $(4,2,2)$.  Conjugation leaves their tuples
unchanged: the sign character is $1$ on $(2^4)$ and $(4^2)$, and
$\chi^\lambda_{(8)}=0$.

Every row of \eqref{eq:small-boolean-coefficients} satisfies
\[
 \gamma_{\lambda,1}
 +\sum_{\substack{d\mid n\\d>1}}\min\{0,\gamma_{\lambda,d}\}\ge0,
\]
with strict inequality in degree $8$.  Since $1\in J$, this expression
is a lower bound for $[s_\lambda]Q_{n,J}$.
This proves all the remaining assertions.
\end{proof}

For composite $n<18$, the proposition and the distinguished coefficients
show that the minimum Schur coefficient over nonempty admissible
supports is $0$, attained for $J=\{1\}$ at $(n)$.
For exceptional supports in even degree, the minimum outside $(1^n)$
is $0$, attained at $(n-1,1)$.

\Needspace{10\baselineskip}
\subsection{Ramanujan squares}\label{subsec:ramanujan-low}

\begin{proposition}\label{prop:ramanujan-low}
For $1\le n<18$, all Schur coefficients of $R_n^{\sq}$ are positive,
except that the coefficient of $s_{(1^n)}$ is zero when
$n\equiv2\pmod4$.
The minimum Schur coefficient is given by
\[
\begin{array}{c|ccccc}
n&1,3&4,5,7,8,11,13,17&9,12,15&16&2,6,10,14\\ \hline
\min_{\lambda\vdash n}[s_\lambda]R_n^{\sq}&1&2&4&6&0.
\end{array}
\]
\end{proposition}

\begin{proof}
The cases $n=1,2,3$ follow from
\[
 R_1^{\sq}=s_{(1)},\qquad
 R_2^{\sq}=2s_{(2)},\qquad
 R_3^{\sq}=2s_{(3)}+s_{(2,1)}+2s_{(1^3)}.
\]
If $4\le n<18$ and $n\notin\{8,9,16\}$, then $n$ is squarefree
or four times an odd squarefree integer.
By \cite[Theorem~36]{ShareshianSundaram}, there is an expansion
\[
 R_n^{\sq}=\sum_{e\mid n}b_e\ell_n^{(e)},\qquad b_e\ge0.
\]
Comparison of the coefficients of $p_1^n$ gives $\sum_{e\mid n}b_e=n$.
Swanson's support theorem \cite[Theorem~1.5]{Swanson} and \eqref{eq:KW}
give $[s_\lambda]\ell_n^{(e)}\ge1$ for every $e\mid n$ whenever
$\lambda$ lies outside the four distinguished partitions, with the
additional exceptions
$(2,2)$, $(3,3)$, and $(2,2,2)$.
Thus $[s_\lambda]R_n^{\sq}\ge n$ away from these exceptions.
For the three exceptional shapes, \eqref{eq:MN-border-strip-appendix}
gives
\[
 \begin{aligned}
 [s_{(2,2)}]R_4^{\sq}&=2+2=4,\\
 [s_{(3,3)}]R_6^{\sq}&=5-3+2=4,\\
 [s_{(2^3)}]R_6^{\sq}&=5+3+2=10.
 \end{aligned}
\]
The endpoint formulas in
\Cref{lem:endpoint-coefficients,lem:local-tn,lem:ramanujan-endpoints}
now give the stated minima and all zeros in these degrees.

For $n=8,16$, the local Ramanujan formula gives
\[
 R_n^{\sq}=p_1^n+p_2^{n/2}+4p_4^{n/4}.
\]
The long-row weight bounds in \eqref{eq:small-long-row-bounds}
therefore apply.  For the weight and its sign twist,
$n\tau_n(w_n)-1=t(n)-1=5$.
If $\lambda$ has a long first row or column and is not one of the four
distinguished partitions, the argument leading to
\eqref{eq:small-long-row-coefficient} gives the following lower bound
for $[s_\lambda]R_n^{\sq}$:
\[
 \frac45f^\lambda-(t(n)-1).
\]
Here $f^\lambda\ge20$ for $n=8$ and $f^\lambda\ge104$ for $n=16$,
by \Cref{lem:long-row-dimension} and \eqref{eq:grn}.
The resulting lower bounds are $11>2$ and $391/5>6$, respectively.

For balanced partitions of $16$,
\eqref{eq:small-balanced-bounds} gives
\[
 \frac{|\chi^\lambda_{(2^8)}|+4|\chi^\lambda_{(4^4)}|}
      {f^\lambda}
 <\frac1{15}+\frac{12}{235}=\frac{83}{705}<\frac18.
\]
Since $f^\lambda\ge1430$, the coefficient of $R_{16}^{\sq}$ is
greater than $7\cdot1430/8>6$.
For balanced partitions of $8$, \eqref{eq:small-gamma-character} gives
\[
 [s_\lambda]R_8^{\sq}
 =8\gamma_{\lambda,1}+8\gamma_{\lambda,2}+12\gamma_{\lambda,4}.
\]
The five degree-$8$ rows of \eqref{eq:small-boolean-coefficients}
therefore give, in order, $28,60,64,92,44$; conjugate partitions give
the same values.  Finally, the four distinguished coefficients are
$6,2,2,6$ in degree $8$ and $6,10,10,6$ in degree $16$.
This proves the assertions in these two degrees.

It remains to consider $n=9$.  Since $R_9^{\sq}=p_1^9+4p_3^3$,
\[
 [s_\lambda]R_9^{\sq}=f^\lambda+4\chi^\lambda_{(3^3)}.
\]
The character values in \eqref{eq:degree-nine-coefficients} give a
minimum of $4$, attained at $(8,1)$ and its conjugate, and show that
every coefficient is positive.
\end{proof}

\section*{Acknowledgments}
The author was supported by the National Research Foundation of Korea
(NRF), funded by the Korean government (Ministry of Science and ICT;
grant RS-2024-00342349).

\end{document}